\documentclass[sn-mathphys-num]{sn-jnl}% Math and Physical Sciences Numbered Reference Style 
\usepackage[utf8]{inputenc}
\usepackage{graphicx}%
\usepackage{multirow}%
\usepackage{amsmath,amssymb,amsfonts}%
\usepackage{amsthm}%
\usepackage{mathrsfs}%
\usepackage[title]{appendix}%
\usepackage{xcolor}%
\usepackage{textcomp}%
\usepackage{manyfoot}%
\usepackage{booktabs}%
\usepackage{algorithm}%
\usepackage{algorithmicx}%
\usepackage{algpseudocode}%
\usepackage{listings}%
\usepackage{resizegather}
\usepackage{subfig}
\usepackage[shortlabels]{enumitem}

\usepackage{pifont}
\theoremstyle{thmstyleone}%
\newtheorem{theorem}{Theorem}%  meant for continuous numbers
\newtheorem{proposition}[theorem]{Proposition}% 
\newtheorem{lemma}[theorem]{Lemma}% 

\theoremstyle{thmstyletwo}%
\newtheorem{assumption}{Assumption}
\newtheorem{remark}{Remark}%

\theoremstyle{definition}%
\newtheorem{definition}{Definition}%

\DeclareMathOperator*{\argmin}{\arg\min}

\newcommand{\rev}[1]{{#1}}

\begin{document}

\title[On the Computation of the Efficient Frontier in Advanced Sparse Portfolio Optimization]{On the Computation of the Efficient Frontier in Advanced Sparse Portfolio Optimization}

%%=============================================================%%
%% GivenName	-> \fnm{Joergen W.}
%% Particle	-> \spfx{van der} -> surname prefix
%% FamilyName	-> \sur{Ploeg}
%% Suffix	-> \sfx{IV}
%% \author*[1,2]{\fnm{Joergen W.} \spfx{van der} \sur{Ploeg} 
%%  \sfx{IV}}\email{iauthor@gmail.com}
%%=============================================================%%

\author[1]{\fnm{Arturo} \sur{Annunziata}}\email{arturo.annunziata@unifi.it}

\author[1]{\fnm{Matteo} \sur{Lapucci}}\email{matteo.lapucci@unifi.it}

\author[1]{\fnm{Pierluigi} \sur{Mansueto}}\email{pierluigi.mansueto@unifi.it}

\author*[1]{\fnm{Davide} \sur{Pucci}}\email{davide.pucci@unifi.it}

\affil[1]{\orgdiv{Global Optimization Laboratory -- Department of Information Engineering (DINFO)}, \orgname{University of Florence}, \orgaddress{\street{Via di Santa Marta, 3}, \city{Florence}, \postcode{50139}, \country{Italy}}}

%%==================================%%
%% Sample for unstructured abstract %%
%%==================================%%

\abstract{In this work, we deal with the problem of computing a comprehensive front of efficient solutions in multi-objective portfolio optimization problems in presence of sparsity constraints. We start the discussion pointing out some weaknesses of the classical linear scalarization approach when applied to the considered class of problems. We are then motivated to propose a suitable algorithmic framework that is designed to overcome these limitations: the novel algorithm combines a gradient-based exploration-refinement strategy with a tailored initialization scheme based on memetic or multi-start descent procedures. Thorough computational experiments highlight how the proposed method is far superior to both linear scalarization and popular genetic algorithms.}


%%\pacs[JEL Classification]{D8, H51}


\maketitle

\section{Introduction}

Portfolio selection has undoubtedly constituted one of the most attractive and studied problems in decision science, given its relevance to the finance community. Starting with the pioneering work of Markowitz \cite{Markowitz52}, the task of suitably choosing how to allocate resources among a set of assets in the market has been formalized as an optimization problem, with a growing complexity as details and real-world specifications got taken into consideration. 

\smallskip
From the one hand, the convex quadratic optimization problem with standard simplex constraints associated with the baseline Markowitz model has been further addressed taking into account more complex objectives, such as the Sharpe ratio \cite{Sharpe1998}, the ESG criterion \cite{PEDERSEN2021572}, or portfolio skewness \cite{DEATHAYDE20041335}; while some of these objectives are fairly simple to manage -- for instance, ESG is linear -- other objectives (skewness, Sharpe) even introduce nonconvexity elements within the problem, making it much harder to deal with. \rev{\label{rev2.5}Some of these objectives have also been studied in the context of Robust Optimization \cite{Scutella2013}, where parameter uncertainty -- common in this field -- is explicitly incorporated into the decision-making process.}

Then, we might note that the mean-variance Markowitz model \cite{markowitz1994general} is already an intrinsically bi-objective optimization model balancing two conflicting objectives: maximizing the expected return and minimizing the variance of the portfolio. The trade-off is commonly handled using a scalarized formulation, where a positive coefficient reflects the investor's risk aversion. Different objectives in most cases then come as additions to portfolio mean and variance, rather than as replacements. The growing importance of sustainable investments motivates the inclusion of ESG criteria in portfolio optimization, to quantify the environmental, social, and governance aspects of investments \cite{PEDERSEN2021572, STEUER2023742}. Skewness allows to distinguish between upside potential and downside risk; furthermore, the inclusion of a higher-order moment partly overcomes the unlikeliness of the utility functions depending only on the first two moments of portfolio return's distribution \cite{DEATHAYDE20041335, MANDAL2024121625}. While a weighted sum of the objectives is a practical way to deal with all requirements at once (see, e.g., \cite{Pascoletti1984, Eichfelder09}), many researchers have investigated how to explicitly address the multi-objective formulation \cite{fliege2016, deb02, custodio11, lapucci2024effectivefrontdescentalgorithmsconvergence, COCCHI2021100008, GARCIABERNABEU2024100305, radziukyniene2008evolutionary, Armananzas05, chen2019robust}. Indeed, retrieving a wide set of efficient solutions, i.e., optimal according to Pareto's theory, delegates the final choice of balance between the objectives to the final investor, who is provided with a clear picture of all possible trade-off scenarios to choose from.

\medskip
On the other hand, new constraints allowed to model the task with increased realism: the addition of general linear constraints, for example, allows to take into account no-short-selling constraints, maximum or minimum investment bounds, exposure limits to specific sectors or industries, turnover constraints, and constraints modeling benchmark considerations; for further information, we refer the reader to \cite{cornuejols2018optimization}. A much harder challenge is then posed by the introduction of a sparsity constraint \cite{Bertsimas22, Bertsimas2009, bienstock1996computational} to limit the number of active assets in the portfolio. This constraint is motivated mainly because managing and monitoring a large number of assets is costly, and sparse portfolios are often perceived by investors as reflecting active management strategies \cite{Bertsimas22}. Furthermore, there is a fixed cost associated with each transaction, which further incentivizes sparse portfolios. Sparsity constraints involve thresholds on the zero pseudo-norm, making the feasible set heavily nonconvex and introducing a strong combinatorial element within the problem \cite{bienstock1996computational}.

\rev{\label{rev2.1}Sparse optimization has been extensively studied in the literature, with growing interest in formulations that impose a hard constraint on the number of nonzero variables. This approach has been applied in a range of domains, especially including portfolio optimization \cite{Bertsimas22,gao13,Cesarone2013}, machine learning applications \cite{miller2002subset,bertsimas16,posada22,carreira2018learning,carlini2017towards,BOMZE2025, Borsos20} and signal processing \cite{BLUMENSATH2009265}, and is thoroughly reviewed in \cite{tillmann24}. A key practical advantage of such formulations is that sparsity can be explicitly controlled through an interpretable threshold, which can be directly specified by the end user. In contrast, sparsity-regularized approaches -- particularly those relying on the widely used $\ell_1$-norm as a convex surrogate for the $\ell_0$-norm \cite{tillmann24} -- are known to suffer from several limitations \cite{bertsimas16}: (i) in noisy settings or with correlated predictors, they often retain many irrelevant variables and shrink all coefficients toward zero, introducing bias; (ii) unlike the $\ell_0$ formulation, which selects variables without shrinkage, the $\ell_1$ approximation may exclude relevant predictors as the regularization strength increases; and (iii) when key assumptions about the data are violated, performance deteriorates in terms of both variable selection and predictive accuracy. These issues are particularly problematic in sparsity-constrained contexts, where the interpretability of the sparsity level is lost, making it unclear how to tune the $\ell_1$-norm to impose a specific constraint on the number of nonzero entries.}

A complete, modern formulation of the portfolio selection problem thus consists of a sparse multi-objective optimization (MOO) problem of the form:
\begin{equation}
	\label{eq::port-prob}
	\begin{aligned}
		\min_{x\in\mathbb{R}^n}\;&F(x) = [f_1(x),\ldots,f_m(x)]^\top
		\\\text{s.t. }& Ax\le b, \quad \mathbf{1}_n^\top x = 1 \\&\|x\|_0 \le s, \quad x \ge \mathbf{0}_n,
	\end{aligned}
\end{equation}
where $n$ is the number of assets, $f_1,\ldots,f_m$ are continuously differentiable objective functions, $A\in\mathbb{R}^{p\times n}$, $b\in\mathbb{R}^p$, $\|\cdot\|_0$ denotes the $\ell_0$ pseudo-norm, i.e., the number of non-zero components of a vector, $s \in \mathbb{N}$ such that $1 \le s < n$ and $\mathbf{1}_n$ and $\mathbf{0}_n$ indicate the $n$-dimensional vectors of all ones and all zeros, respectively. The above formulation has been the subject of some studies \cite{Lapucci2022apenalty, Bertsimas22}, but no thorough framework has been proposed on the methodological side to properly tackle it.

In this paper, we deal with this challenge. In the first place, we discuss the weakness of the simple \rev{(linear)} scalarization strategy. Specifically, we point out that the presence of the cardinality constraint breaks, even in the mean-variance case, the property of the problem only having supported solutions: some efficient trade-offs cannot be obtained as the optimal solution of the scalarized problem for any choice of the weights. In fact, entire portions of the Pareto front, corresponding to ``good'' choices of active assets, are somewhat shadowed in this setting. Moreover, the obtainment of a well spanned and uniformly spread front of solutions is further precluded by the complexity of choosing suitable wights.  Inspired by the work in \cite{Lapucci2024cardinality}, we then propose a detailed algorithmic framework specifically designed to effectively retrieve an accurate approximation of the Pareto front in instances of problem \eqref{eq::port-prob}. 

The proposed method exploits a heuristic evolutionary strategy \cite{deb02, Lapucci2023}, equipped with a tailored derivative-based ``hard-thresholding'' procedure \cite{Lapucci2022apenalty, Lapucci2024cardinality}, to feed with good starting points a second-stage multi-objective optimization method, based on the Front descent algorithm \cite{lapucci2024effectivefrontdescentalgorithmsconvergence}, for filling the Pareto front. All the cogs of the proposed methodology have been polished to be compatible with the presence of general linear constraints. By a thorough computational experimentation, we show that the proposed approach consistently solves the problem with high effectiveness, with superior performance w.r.t.\ state-of-the-art approaches.

The rest of the paper is organized as follows. In Section \ref{sec::preliminaries}, we review objective functions and constraints typical of portfolio problems; we also recall basic notions in multi-objective optimization, with some focus on \rev{linear} scalarization approaches. In Section \ref{sec::limits_scalarization}, we show through a simple example the limitations of the latter class approaches in finding specific efficient solutions. Then, in Section \ref{sec::SFSD-based-methods} we propose an algorithmic framework designed to effectively reconstruct the Pareto front of MOO sparse portfolio problems\rev{\label{rev2.4.a}; along with this, Appendix \ref{app::descent_methods} presents schematic and theoretical modifications of some methodologies that can be used within the framework but were not originally designed to handle constraints other than sparsity}. In Section \ref{sec::experiments}, we show the results of thorough computational experiments where we compared our framework with state-of-the-art approaches. Finally, in Section \ref{sec::conclusions} we provide some concluding remarks.

\section{Preliminaries}
\label{sec::preliminaries}

In this preliminary section, we provide a brief description of common objective functions and constraints in portfolio problems. Then, we continue reviewing some basic notions for (sparse) MOO problems.

\subsection{Objective Functions and Constraints in Portfolio Optimization}

In the core discussion of this paper, we are going to consider multi-objective formulations of the (sparse) portfolio selection problem, considering several combinations of objectives. Specifically, we will be interested in:
\begin{description}[font=\textbf]
	\item[Expected Return] The expected value of the return of investment is defined as
	\begin{equation}
		\label{eq::expected_return}
		f^\text{ER}(x) = c^\top x,
	\end{equation}
	with $c \in \mathbb{R}^n$ being the \textit{expected returns vector} \cite{markowitz1994general};  maximizing this quantity is arguably the easiest and most straightforward goal in portfolio optimization.
	\item[Portfolio Risk]  The well-known variance function, defined as 
	\begin{equation}
		\label{eq::variance}
		f^\text{V}(x) = \frac{1}{2}x^\top Qx,
	\end{equation} 
	where $Q \in \mathbb{R}^{n \times n}$ denotes the positive semi-definite \textit{variance-covariance matrix}, is classically used to measure the risk associated with a portfolio. Roughly speaking, the higher is the variance of the return of investment, the higher will be the chances of the realized return to be far lower than the expected value. 
	\item[ESG Score] The \textit{Environmental, Social, and Governance} (ESG) score \cite{PEDERSEN2021572, STEUER2023742} has become an essential criterion in portfolio optimization, reflecting an investor's preference for sustainable and responsible investments. 
	
	For each asset $i \in \{1,\ldots,n\}$ in the investment universe, let us denote as $v_i$ its ESG score, which quantifies its performance on environmental, social, and governance dimensions. Let us denote the vector of ESG scores as $v = [v_1,\ldots,v_n]^\top \in \mathbb{R}^n$. The portfolio's ESG score is then calculated as the weighted sum of the ESG scores of its constituent assets:
	\begin{equation*}
		\label{eq::esg}
		f^\text{ESG}(x) = x^\top v.
	\end{equation*}
	
	Maximizing the ESG score, portfolio managers aim to balance traditional financial objectives (e.g., risk and return) with non-financial considerations, such as sustainability and ethical alignment so as to ensure that the portfolio aligns with the investor's sustainability goals.
	\item[Sharpe Ratio] The Sharpe ratio \cite{Sharpe1998} is a widely used metric that incorporates the expected return and variance functions to evaluate the risk-adjusted performance of a portfolio. Formally, it is defined as
	\begin{equation*}
		f^\text{SR}(x) = \frac{f^\text{ER}(x)}{\sqrt{f^\text{V}(x)}},
	\end{equation*}
	i.e., it is the ratio between the expected value and the standard deviation of the return of investment.
	
	This quantity aims to measure how much return a portfolio generates per unit of risk, with higher values indicating better risk-adjusted performance. 
	In brief, the Sharpe ratio helps investors compare portfolios with different risk profiles, penalizing the ones with high volatility unless accompanied by proportionally higher returns.
	\item[Portfolio Skewness] Skewness represents a measure for the asymmetry of the return distribution \cite{DEATHAYDE20041335, MANDAL2024121625}: a positive value indicates a higher probability of large positive returns, while negative skewness reflects a higher likelihood of extreme losses. Investors often prefer portfolios with positive skewness, as they provide enhanced upside potential while mitigating downside risks. Thus, by considering skewness as objective function, portfolio managers can address tail risks and align portfolios with investor preferences for asymmetrical return profiles.
	
	Mathematically, following the idea in \cite{DEATHAYDE20041335, PedroBrito2016}, the skewness of a portfolio can be formalized as the third moment of portfolio return, i.e.,
	\begin{equation}
		\label{eq::skew}
		f^\text{SW}(x) = \mathbb{E}[(\text{R}_p - f^\text{ER}(x))^3],
	\end{equation}
	where \( R_p \) is the random variable of portfolio return and \( f^\text{ER}(x) \) is the expected return of the portfolio.
	
	In order to calculate equation \eqref{eq::skew}, we need to define the coskewness tensor $\mathbf{C}$, which captures third-order interactions among asset returns. In particular, $\mathbf{C}$ is a tensor of size $n \times n \times n$ whose $\mathbf{C}_{ijk}$ is given by
	\begin{equation*}
		\mathbf{C}_{ijk} = \mathbb{E}[(r_i - c_i)(r_j - c_j)(r_k - c_k)],
	\end{equation*}
	where $r_i$ denotes the return of asset $i$ and $c_i$ is its mean. The expectation is computed across the time series. The value of $\mathbf{C}_{ijk}$ quantifies the asymmetric interaction between the deviations of assets $i$, $j$, and $k$ from their respective means. The portfolio skewness is then computed by combining the coskewness tensor \( \mathbf{C} \) with the vector of portfolio weights $x$:
	\begin{equation*}
		f^\text{SW}(x) = x^\top \mathbf{C} (x \otimes x),
	\end{equation*}
	where $\otimes$ indicates the Kronecker product.
\end{description}

\medskip 
As for the constraints, the general portfolio model we address in this work possibly encompasses the following specifications \cite{cornuejols2018optimization}:
\begin{description}[font=\textbf]
	\item[Standard Simplex Constraints] 
	The standard simplex constraint is a basic requirement in portfolio optimization. Mathematically, this constraint is defined as:
	\begin{equation*}
		\label{eq::simplex}
		x \geq \mathbf{0}_n, \qquad \mathbf{1}_n^\top x = 1.
	\end{equation*}
	
	The first condition prevents short-selling, meaning that the budget  $x_i$ allocated to each asset \( i \in \{1, \ldots, n\} \) shall be non-negative. 
	
	The second condition ensures  that an investor allocates all the available capital to the given set of assets, without leaving any budget portion uninvested.

	\item[Minimum and Maximum Investment Constraints]
	
	Minimum investment constraints ensure that a certain minimum portion of the total capital is allocated to specific assets, while maximum investment constraints prevent concentration in any single asset. These constraints can be expressed as:
	\[
	\ell_i \leq x_i \leq u_i, \quad \forall i \in \{1, \ldots, n\},
	\]
	where $\ell_i \in\mathbb{R}$ and $u_i \in\mathbb{R}$ such that $l_i\le u_i$ are the minimum and the maximum investment bounds for asset \( i \).
	In a compact representation, this set of constraints becomes
	\[
	\ell \leq x \leq u,
	\]
	where $ \ell, u \in \mathbb{R}^n$ are the vectors of minimum and maximum investment bounds respectively.
	
	\item[Portfolio Beta Constraints] The beta coefficient measures the sensitivity of an asset return to movements in a market index. Formally, for an individual asset $i\in \{1,\ldots, n\}$, the beta is defined as
	\[
	\beta_i = \frac{\text{Cov}(r_i, r_m)}{\text{Var}(r_m)},
	\]
	where \( r_i \) and \( r_m \) are random variables, respectively the return of asset \( i \) and the return of the market index, \(\text{Cov}(r_i, r_m)\) is the covariance between \( r_i \) and \( r_m \), \(\text{Var}(r_m)\) is the variance of the market index return.
	Let us define the vector $\beta=[\beta_1,\ldots, \beta_n]^\top \in \mathbb{R}^n$.
	
	In practical applications, investors may seek to limit the portfolio beta within a predefined range to manage exposure to market risk; thus, we may want to impose lower and upper bound constraints on the weighted sum of the betas:
	\begin{gather*}
		x^\top\beta \ge \beta_\text{min}\\
		x^\top\beta \le \beta_\text{max},
	\end{gather*}
	with $\beta_\text{min}, \beta_\text{max} \in \mathbb{R}_+$,  $\beta_\text{min} \le \beta_\text{max}$.
	In order to embed these linear constraints in formulation \eqref{eq::port-prob}, we just need to set 
	\begin{equation*}
		A = \begin{bmatrix} -\beta \\ \beta \end{bmatrix} \qquad b = \begin{bmatrix} -\beta_\text{min} \\ \beta_\text{max} \end{bmatrix}.
	\end{equation*}

	\item[Sector or Industry Exposure Constraints]
	
	Another common requirement in portfolio management is to limit the exposure to specific sectors or industries. Given a set $\mathcal{I}$ representing the set of assets belonging to a given industry or sector, the following linear constraint can be imposed
	\[
	\eta_\mathcal{I} \leq \sum_{j\in \mathcal{I}} x_j \leq \mu_\mathcal{I}
	\]
	where $\eta_\mathcal{I},\mu_\mathcal{I} \in\mathbb{R}$, with $\eta_\mathcal{I}\le \mu_\mathcal{I}$, are the minimum and maximum exposure respectively to the considered industry or sector. 
	
	\item[Turnover Constraint]
	Given an initial allocation vector $x^0 \in \mathbb{R}^n$, the turnover constraint limits the changes in portfolio composition in the following period, reducing transaction costs. This kind of constraint is defined as
	\[
	\sum_{j=1}^n |x_j^0 - x_j| \leq \tau,
	\]
	where \( \tau \) is the maximum total turnover allowed.
	The above nonlinear formulation can of course  be rewritten equivalently as a set of linear constraints.  In particular, introducing a new vector of auxiliary variables $y\in\mathbb{R}^n$, we can impose
	\begin{equation*}
		-y_i\le x_i - x_i^0 \leq y_i \quad \forall\, i=1,\ldots,n,\qquad\text{ and }\qquad
		\sum_{i=1}^n y_i \leq \tau.
	\end{equation*}
	
	\item[Cardinality Constraint] The cardinality constraint is a requirement introduced to limit the number of assets actively included in the portfolio. This constraint is motivated by practical considerations, such as the cost and complexity of managing a large number of assets, as well as the investor's preference for a more concentrated and actively managed portfolio. Sparse portfolios are also more efficient due to reduced transaction costs.
	Formally, the constraint is imposed setting an upper bound $s<n$ to the \( \ell_0 \)-pseudo-norm of the portfolio allocation vector \( x \), i.e.,
	\[
	\|x\|_0 \leq s.
	\]

\end{description}

\subsection{Fundamentals of Multi-objective Optimization}
\label{subsec::moo_basics}

Let us consider a multi-objective optimization problem of the form
\begin{equation}
	\label{eq:gen_mo_prob}
	\begin{aligned}
		\min_{x\in\mathbb{R}^n}\;&F(x) = \left[f_1(x),\ldots,f_m(x)\right]^\top\\
		\text{s.t. }& x\in \Omega,
	\end{aligned}
\end{equation}
where $\Omega$ is some nonempty closed feasible set. 
To compare solutions of problem \eqref{eq:gen_mo_prob} to each other, we make use of the standard partial ordering on $\mathbb{R}^m$: given $u, v \in \mathbb{R}^m$, $u \le v$ means $u_i \le v_i$ for all $i \in \{1,\ldots, m\}$; similarly, $u < v$ if $u_i < v_i$ for all $i$; finally, we denote by $u \lneqq v$ when $u \le v$ and $u \ne v$. 
The latter relation allows to introduce the notion of dominance between solutions of problem \eqref{eq:gen_mo_prob}: given $x, y \in \Omega$, we say that $y$ is \textit{dominated} by $x$ if $F(x) \lneqq F(y)$. We are now able to introduce Pareto optimality concepts for problem \eqref{eq:gen_mo_prob}.
\begin{definition}
	A solution $\bar{x} \in \Omega$ is:
	\begin{enumerate}
		\item \textit{Pareto optimal} (or \textit{efficient}) for problem \eqref{eq:gen_mo_prob} if there is no $y \in \Omega$ such that $F(y) \lneqq F(\bar{x})$;
		\item \textit{weakly Pareto optimal} (or \textit{weakly efficient}) for problem \eqref{eq:gen_mo_prob} if there is no $y \in \Omega$ such that $F(y) < F(\bar{x})$.
	\end{enumerate}
	If one of the above properties holds in $\Omega \cap \mathcal{N}(\bar{x})$ for some neighborhood $\mathcal{N}(\bar{x})$ of $\bar{x}$, then $\bar{x}$ is locally (weakly) Pareto optimal.
\end{definition}

Clearly, Pareto optimality is a stronger property and implies weak Pareto optimality. The set of Pareto optimal solutions is called \textit{Pareto set} (or \textit{efficient set}), while the image of the latter through $F(\cdot)$ is named \textit{Pareto front} (or \textit{efficient frontier}).
Under differentiability assumptions on $F$, first-order necessary conditions for Pareto optimality can also be given.
\begin{lemma}[{\cite[\rev{Proposition 3.3}]{Lapucci2022apenalty}}]
	\label{lem::par-stat}
	Let point $\bar{x} \in \Omega$ be a (weakly) efficient solution of \eqref{eq:gen_mo_prob}; then $\bar{x}$ is Pareto stationary for problem \eqref{eq:gen_mo_prob}, i.e., 
	\begin{equation}
		\label{eq:stat_prob}
		\min_{d \in \mathcal{D}_\Omega(\bar{x})}\max_{j \in \{1,\ldots, m\}} \nabla f_j(\bar{x})^\top d + \frac{1}{2}\|d\|^2 = 0,
	\end{equation}
	where $\mathcal{D}_\Omega(\bar{x}) = \{v \in \mathbb{R}^n \mid \exists\, \bar{t} > 0: \bar{x} + tv \in \Omega \; \forall t \in [0, \bar{t}\,]\}$ is the set of feasible directions at $\bar{x}$.
\end{lemma}
\noindent We will denote by $\theta(\bar{x})$ and $v(\bar{x})$ the minimum and optimal solution of problem \eqref{eq:stat_prob}, respectively. Note that, when $\Omega$ is non-convex, $v(\bar{x})$ may not be unique.

\medskip
The easiest and arguably most common practice to tackle problem \eqref{eq:gen_mo_prob} consists in resorting to the weighted sum scalarization method \cite{Eichfelder09, ehrgott2005multicriteria}. The focus then shifts to solving the single-objective problem  
\begin{equation}
	\label{eq::port-prob-scal}
	P(\lambda) = \min_{x\in\Omega}\;\sum_{j=1}^{m} \lambda_j f_j(x),
\end{equation}
where the weights $\lambda = [\lambda_1, \dots, \lambda_m]^\top \in \mathbb{R}^m$ are such that $\lambda_j \ge 0$ for all $j \in \{1,\ldots,m\}$ and $\sum_{j=1}^m\lambda_j = 1$. If we consider, for instance, the variance $f^\text{V}$ and the expected return $f^\text{ER}$ as objectives of the portfolio problem, we immediately retrieve the original Markovitz formulation \cite{markowitz1994general, Markowitz52}.

The relationship between optimal solutions of a scalarized problem and efficient solutions of the multi-objective problem can be characterized by the following concept.
\begin{definition}[{\cite[Definition 8.7]{ehrgott2005multicriteria}}]
	Let $\bar{x}$ be an efficient solution of problem \eqref{eq:gen_mo_prob}. If there exists some $\lambda \in \mathbb{R}^m$ such that $\bar{x}$ is optimal for problem $P(\lambda)$ defined in \eqref{eq::port-prob-scal}, then $\bar{x}$ is called a \textit{supported efficient solution}, while $y = F(x)$ is called a \textit{supported non-dominated point}. 
\end{definition}
It is easy to prove that, for any choice of $\lambda\ge \mathbf{0}_m$, every globally optimal solution of $P(\lambda)$ constitutes a Pareto optimal point for problem \eqref{eq:gen_mo_prob} and, thus, a \textit{supported} efficient solution. The converse is in general not necessarily true. In these particular cases, \rev{\label{rev2.2.a}it is possible to rely on non-linear scalarizing functions, such as the weighted Tchebycheff norm \cite{bowman1976relationship,steuer1983interactive}. However, although theoretically powerful -- being capable of reaching any Pareto-optimal trade-off with a suitable choice of weights -- these approaches often lead to scalarized subproblems that are challenging to solve to global optimality. When such subproblems are tractable, }scalarization might be seen as a sufficient strategy to deal with the multi-objective problem\rev{; nevertheless,} choosing the weights for obtaining a well spread Pareto front \rev{still} remains a nontrivial task.

\subsection{\rev{Additional Notation for Cardinality Constraints}}
\label{subsec::additional_notation_cc}

In the following sections, we will consider the case where the cardinality constraint $\|x\|_0\le s$ is included in the definition of the feasible set $\Omega$. To deal with this setting, we need to introduce some additional notation. Given an index set $S \subseteq \{1,\ldots, n\}$, its cardinality is indicated with $|S|$, while we denote by $\bar{S} = \{1,\ldots, n\} \setminus S$ its complementary set. Given $x \in \Omega$, we indicate by $x_S$ the sub-vector of $x$ induced by $S$, i.e., the vector composed by the components $x_i$, with $i \in S$. Then, we define the \textit{support set} of $x$, i.e., the set of the indices corresponding to the nonzero components of $x$, as $S_1(x) = \{i \in \{1,\ldots, n\} \mid x_i \ne 0\}$; $S_0(x) = \{1,\ldots, n\} \setminus S_1(x)$ is the complementary set of $S_1(x)$. Finally, according to \cite{beck2016}, we say that an index set $J$ is a super support set for $x$ if $S_1(x) \subseteq J$ and $|J| = s$; the set of all super support sets at $x$ is denoted by $\mathcal{J}(x)$ and $|\mathcal{J}(x)| = 1$ if and only if $\|x\|_0 = s$.

\section{Limitations of Scalarization in Cardinality-constrained Problems}
\label{sec::limits_scalarization}
Although scalarization methods have arguably been the most employed among the approaches for portfolio problems (see, e.g., \cite{markowitz1994general, Markowitz52, markowitz2000mean, hm1959portfolio, cesarone2020optimization}), they become provably faulty as the problems get less regular and, in particular, when there are non-supported efficient solutions; this is the case of cardinality-constrained portfolio problems \eqref{eq::port-prob-scal}. Indeed, even with convex objective functions, such as variance $f^\text{V}$ \eqref{eq::variance} and expected return $f^\text{ER}$ \eqref{eq::expected_return}, the presence of the cardinality constraint possibly makes problem \eqref{eq::port-prob} not supported anymore. 

In order to illustrate this downside of the \rev{linear} scalarization approach, let us consider a simple three-dimensional example with variance and expected return as objective functions:
\begin{equation}
	\label{eq::port-prob-toy}
	\begin{aligned}
		\min_{x\in\mathbb{R}^3}\;&F(x) = [2x_1^2 + 0.5x_2^2 + 3x_3^2, 4x_1 + 5x_2 + x_3]^\top
		\\\text{s.t. }& \mathbf{1}_3^\top x = 1, \quad \|x\|_0 \le 1, \quad x \ge \mathbf{0}_3.
	\end{aligned}
\end{equation}
The above problem can be seen as a mean-variance portfolio problem with cardinality constraint in multi-objective form, setting $$Q = \begin{bmatrix} 2 & 0 & 0 \\ 0 & 0.5 & 0 \\ 0 & 0 & 3 \end{bmatrix},\qquad c = [4, 5, 1]^\top.$$  The corresponding weighted-sum scalarization problem is 
\begin{equation}
	\label{eq::port-prob-toy-scal}
	\begin{aligned}
		\min_{x\in\mathbb{R}^3}\;&2x_1^2 + 0.5x_2^2 + 3x_3^2 + \tilde{\lambda}(4x_1 + 5x_2 + x_3)
		\\\text{s.t. }& \mathbf{1}_3^\top x = 1, \quad \|x\|_0 \le 1, \quad x \ge \mathbf{0}_3,
	\end{aligned}
\end{equation} 
where for simplicity and with no loss of generality we consider a single weight $\tilde{\lambda}$ associated with the second objective function.

It is trivial to see that the combination of the cardinality ($s=1$) and simplex constraints induces a feasible set containing only three solutions: 
$x_1^\star = [1, 0, 0]^\top$, $x_2^\star = [0, 1, 0]^\top$ and $x_3^\star = [0, 0, 1]^\top$; the images of these solutions through $F$ are shown in Figure \ref{fig::pareto}. From problem \eqref{eq::port-prob-toy-scal} we can then define the following three $\tilde{\lambda}$-dependent functions, each one related to one of the three solutions:
$$f_{x_1}(\tilde{\lambda}) = 2 + 4\tilde{\lambda}, \qquad f_{x_2}(\tilde{\lambda}) = 0.5 + 5\tilde{\lambda}, \qquad f_{x_3}(\tilde{\lambda}) = 3 + \tilde{\lambda};$$ the latter ones are plotted for $\tilde{\lambda} \ge 0$ in Figure \ref{fig::functions}.

\begin{figure}
	\subfloat[Pareto optimal solutions of problem \eqref{eq::port-prob-toy}. Solution $x_1^*$ is not a global minimizer for the scalarized problem for any value of $\tilde{\lambda}$; thus, $F(x_1^\star)$ cannot be obtained using the weighted-sum scalarization method.\label{fig::pareto}]{\includegraphics[width=0.39\textwidth]{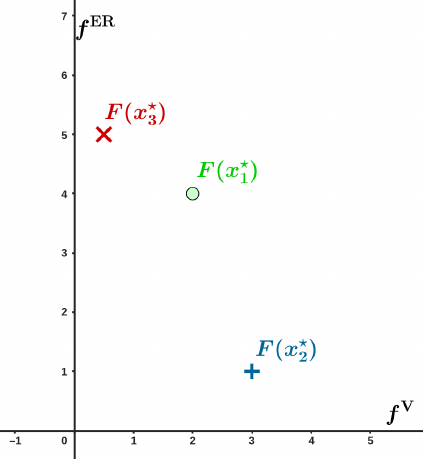}}
	\hfil
	\subfloat[Plot of functions $f_{x_1}(\tilde{\lambda})$, $f_{x_2}(\tilde{\lambda})$ and $f_{x_3}(\tilde{\lambda})$; note that $\nexists\tilde{\lambda} \ge 0$ s.t.\ $f_{x_1}(\tilde{\lambda}) \le \min\{f_{x_2}(\tilde{\lambda}), f_{x_3}(\tilde{\lambda})\}$.\label{fig::functions}]{\includegraphics[width=0.39\textwidth]{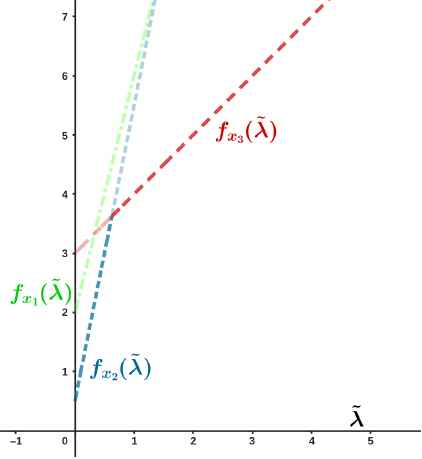}}
	\caption{Example showing the limitations of weighted-sum scalarization approach.}
\end{figure}

Functions $f_{x_1}(\tilde{\lambda})$, $f_{x_2}(\tilde{\lambda})$, $f_{x_3}(\tilde{\lambda})$ give us an immediate intuition of the solutions found given a value for $\tilde{\lambda}$. For $0 \le \tilde{\lambda} < \frac{5}{8}$, where $\tilde{\lambda} = \frac{5}{8}$ is the intersection point of lines $f_{x_2}(\tilde{\lambda})$ and $f_{x_3}(\tilde{\lambda})$, the minimum value is attained by $f_{x_2}(\tilde{\lambda})$, i.e., $f_{x_2}(\tilde{\lambda}) < \min\{f_{x_1}(\tilde{\lambda}), f_{x_3}(\tilde{\lambda})\}$; thus, the best solution is represented by $x_2^\star$. For $\tilde{\lambda}=\frac{5}{8}$, we have $f_{x_2}(\tilde{\lambda}) =f_{x_3}(\tilde{\lambda})<f_{x_1}(\tilde{\lambda})$: $x_2^\star$ and $x_3^\star$ are the two global minimizers. For $\tilde{\lambda} \ge \frac{5}{8}$, we get that $x_3^\star$ is the optimal solution of the scalarized problem. The solution $x_1^\star$ is thus never obtained through \rev{linear} scalarization. However, we can immediately observe from Figure \ref{fig::pareto} that $F(x_1^\star)$ is not dominated and thus $x_1^\star$ actually is an efficient solution.

The possible failure at obtaining some efficient solutions gets exacerbated with much more complex problems. A particularly troublesome situation occurs when \rev{linear} scalarization is unable to detect any of the efficient solutions associated with a particular support set. \rev{\label{rev2.2.b}As already mentioned in Section \ref{subsec::moo_basics}, one possible approach is to employ non-linear scalarizations (e.g., weighted Tchebycheff norms \cite{bowman1976relationship,steuer1983interactive}), which are theoretically capable of reaching any Pareto-optimal solution given suitable choices of weights. Still, there are no predefined rules for selecting these weights, and the resulting scalarized problems are often difficult to solve to global optimality, even with modern solvers. All these issues} thus represent a strong motivation to devise alternative strategies to scalarization when dealing with cardinality-constrained portfolio problems. In the rest of the paper, we will describe a new methodology aimed at fully reconstructing the frontier of efficient solutions in portfolio optimization problems.

\section{A Tailored Sparse Front Descent Method for Portfolio Problems}
\label{sec::SFSD-based-methods}

\rev{In this section, we present a tailored algorithmic framework for cardinality-constrained portfolio problems. The real motif of our proposal lies in the observation that, unlike the smoother Pareto fronts typical of standard nonlinear continuous problems, the Pareto frontier in cardinality-constrained multi-objective optimization tends to be more irregular and composed of multiple smooth segments, each of which corresponds to Pareto-optimal solutions sharing the same support set. Figure~\ref{fig::prototypes} shows typical examples of such Pareto frontiers.}

\begin{figure}
    \centering
    \subfloat{\includegraphics[width=0.49\linewidth]{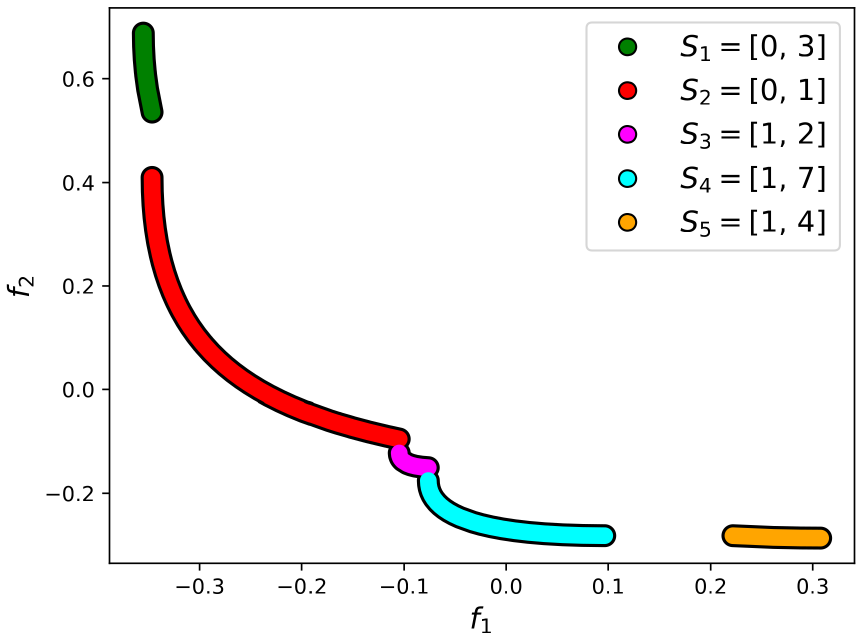}}
    \hfil
    \subfloat{\includegraphics[width=0.49\linewidth]{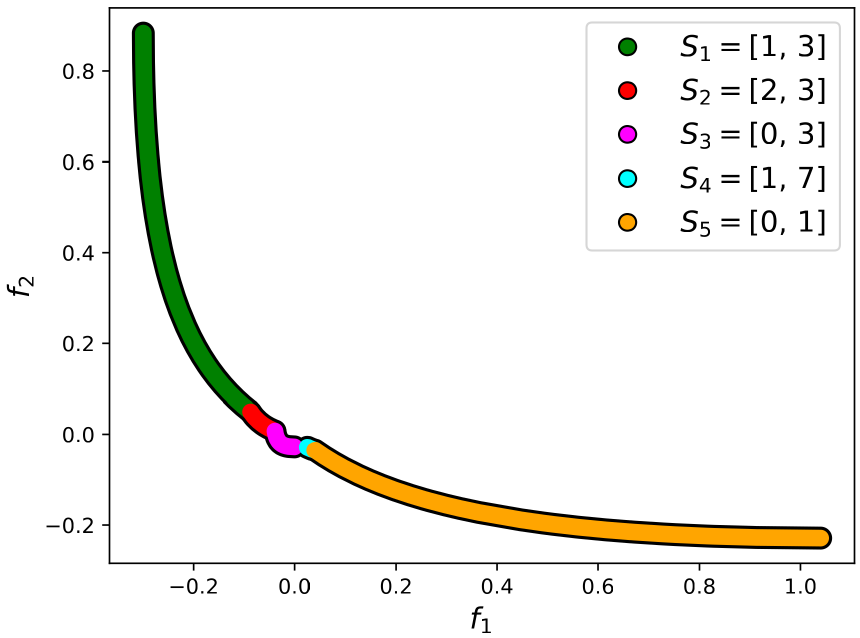}}
    \caption{\rev{Illustration of typical Pareto frontiers for cardinality-constrained multi-objective optimization problems.}}
    \label{fig::prototypes}
\end{figure}

\rev{Our framework then addresses the two main challenges posed by the (more) complex shape of the Pareto frontier:
\begin{enumerate}[i)]
    \item identifying efficient solutions associated with a diverse set of support sets -- ideally those representing distinct segments of the actual Pareto front;
    \item effectively exploring each segment of the Pareto front corresponding to a given support set identified in the first phase.
\end{enumerate}}

\rev{In the next two subsections, we address these challenges separately, beginning with the second.}

\subsection{Sparse Front Steepest Descent}
\label{subsec::SFSD}

In the literature, the methodologies considered to solve problems of the form \eqref{eq::port-prob} consist in scalarization \cite{markowitz1994general, Eichfelder09,Pascoletti1984,fliege2004gap,drummond2008choice,bowman1976relationship,steuer1983interactive} and evolutionary \cite{Chen2017,deb02,laumanns2002combining,mostaghim2007multi,konak2006multi} methods. However, both the two classes of approaches have drawbacks regarding the spanning of the Pareto front; in particular, Pareto front approximations are obtained whose spread highly depends on the (fixed) number of solutions; moreover, the selection of suitable weights to obtain a uniform front is not trivial in scalarization approaches, whereas evolutionary methods typically scale poorly as the dimensions of the problem grow \cite{COCCHI2021100008,Cocchi2020,custodio11}.  In this work, we thus propose a framework revolving around the employment of the \textit{Sparse Front Steepest Descent} (\texttt{SFSD}) algorithm  \cite{Lapucci2024cardinality}; for the case of MOO problems without additional constraints other than sparsity, this method  was shown to be highly effective at retrieving accurate and well-spread Pareto front approximations, provided a suitable strategy is available to find starting solutions with various and good support sets.

\begin{algorithm}
	\caption{\textit{Sparse Front Steepest Descent with Linear Constraints} (\texttt{SFSD})}\label{alg::SFSD}
	\begin{algorithmic}[1]
		\Require $F:\mathbb{R}^n \rightarrow \mathbb{R}^m$, $\delta \in (0, 1)$, $\gamma \in (0, 1)$.
		\State $\mathcal{X}^0$ = \texttt{Initialize}($F$) \label{line::initialization}
		\State $k = 0$
		\While{a stopping criterion is not satisfied}
		\State $\widehat{\mathcal{X}}^k = \mathcal{X}^k$
		\ForAll{$(x_c, J_{x_c})\in \mathcal{X}^k$}
		\If{$(x_c, J_{x_c}) \in \widehat{\mathcal{X}}^k$ }
		\If{$\theta_{J_{x_c}}(x_c)<0$ \label{line::theta_1}}
		\vskip0.1cm
		\State $\alpha_c^k = \max\limits_{h\in\mathbb{N}} \{\delta^h\mid F(x_c+\delta^h v_{J_{x_c}}(x_c))\le F(x_c)+\boldsymbol{1}_m\gamma\delta^h\theta_{J_{x_c}}(x_c)\}$ \label{line::Armijo_1}
		\Else
		\State $\alpha_c^k=0$
		\EndIf 
		\State $z_c^k = x_c+\alpha_c^kv_{J_{x_c}}(x_c)$ \label{line::new_zck}
		\State $\widehat{\mathcal{X}}^k = \widehat{\mathcal{X}}^k \setminus \left\{(y, J_y) \in \widehat{\mathcal{X}}^k \mid J_y = J_{x_c},\; F(z^k_c) \lneqq F(y)\right\} \cup \left\{(z^k_c, J_{x_c})\right\}$ \label{line::filtering-1}
		\ForAll{$\mathcal{I}\subseteq\{1,\ldots,m\}$ s.t.\ $\theta^\mathcal{I}_{J_{x_c}}(z_c^k) < 0$ \label{line::theta_2}}
		\If{$(z_c^k, J_{x_c})\in\widehat{\mathcal{X}}^k$ \label{line::nondom}}
		\State 
		\vskip-0.7cm
		\begin{equation*}
			\quad\quad
			\begin{aligned}
				\alpha_c^\mathcal{I} & = \max_{h \in \mathbb{N}}\left\{\delta^h
				\middle\vert
				\begin{aligned}
					& \forall(y, J_y)\in\widehat{\mathcal{X}}^k, \; J_y=J_{x_c},\; \exists j \text{ s.t. }\\
					&\quad f_j(z_c^k+\delta^h v^\mathcal{I}_{J_{x_c}}(z_c^k)) < f_j(y)
				\end{aligned}
				\right\}
			\end{aligned}
		\end{equation*} \label{line::Armijo_2}
		\State $\hat{z} = z_c^k + \alpha_c^\mathcal{I} v^\mathcal{I}_{J_{x_c}}(z_c^k)$ \label{line::hatz}
		\State $\widehat{\mathcal{X}}^k = \widehat{\mathcal{X}}^k \setminus \left\{(y, J_{y}) \in \widehat{\mathcal{X}}^k \mid J_y=J_{x_c},\; F(\hat{z}) \lneqq F(y)\right\} \cup \left\{(\hat{z}, J_{x_c})\right\}$ \label{line::filtering-2}
		\EndIf
		\EndFor
		\EndIf	
		\EndFor
		\State $\mathcal{X}^{k + 1} = \widehat{\mathcal{X}}^k$
		\State $k = k + 1$
		\EndWhile\\
		\Return $\mathcal{X}^k$
	\end{algorithmic}
\end{algorithm}

The pseudocode of an \texttt{SFSD} procedure, adapted to handle the additional linear constraints commonly appearing in portfolio problems, is reported in Algorithm \ref{alg::SFSD}. The method starts with an initial phase (line \ref{line::initialization}) - crucial for the computational performance - where feasible solutions associated with promising supports, \rev{i.e., support sets possibly associated with efficient solutions}, are provided. \rev{\label{rev2_Jx}Each solution is then paired with its support set, so that the search steps performed by \texttt{SFSD} from that solution contribute to spanning the portion of the Pareto frontier composed of points sharing the same support. If the support set is not full, the solution is instead paired with one of its super support sets to explore a broader region of the feasible space and, when necessary, to eventually reach a full support set. The resulting set for \texttt{SFSD} after the initial phase is thus}

\begin{equation}
    \label{eq::X0}
    \mathcal{X}^0 = \{(x, J_x) \mid x \in \Omega,\; J_x \in \mathcal{J}(x)\}.
\end{equation}

\rev{Note that, by the definition of support set and super support set given in Section \ref{subsec::additional_notation_cc}, if \(|S_1(x)| = s\), then \(\mathcal{J}(x) = \{S_1(x)\}\); that is, solution \(x\) is effectively associated with its support set in the set \(\mathcal{X}^0\).}

\rev{At} a generic iteration $k$, each point $x_c$ in the current list of solutions is processed as follows.
\begin{itemize}
	\item All the objectives are first decreased by an optimization step carried out in the active subspace induced by the support $J_{x_c}$ (lines \ref{line::theta_1}-\ref{line::new_zck}). In particular, a feasible direction of descent for all the objectives is computed (line \ref{line::theta_1}) solving an instance of problem
	\begin{equation}
		\label{eq::common-dir}
		\begin{aligned}
			\theta_{J_{x_c}}(x_c) = \min_{d \in \mathbb{R}^n}\;&\max_{j \in \{1,\ldots, m\}} \nabla f_j(x_c)^\top d + \frac{1}{2}\|d\|^2
			\\\text{s.t. }& A(x_c+d) \le b, \quad \mathbf{1}_n^\top(x_c+d) = 1 \\& d_{\bar{J}_{x_c}} = \mathbf{0}_{|\bar{J}_{x_c}|}, \quad x_c+d \ge \mathbf{0}_n.
		\end{aligned}
	\end{equation}
	We denote by $v_{J_{x_c}}(x_c)$ the problem solution, which we call \textit{constrained common descent direction}. Note that, given the strong convexity of the objective function and the convexity of the feasible set of \eqref{eq::common-dir}, $v_{J_{x_c}}(x_c)$ is unique. An Armijo-type line search is employed (line \ref{line::Armijo_1}) to compute a corresponding stepsize that provides a sufficient decrease of $F$, leading to a new solution $z_c^k$.
	\item As long as $z_c^k$ is not dominated by any other point $y$ in the current list of solutions such that $J_y = J_{x_c}$ (line \ref{line::nondom}), additional line searches in the subspace are carried out starting at $z_c^k$ itself, this time only considering subsets of objectives $\mathcal{I} \subseteq \{1,\ldots,m\}$ (lines \ref{line::theta_2}-\ref{line::hatz}).  The resulting points are sought to be simply non-dominated by any other point in the current set (line \ref{line::Armijo_2}). These ``exploration'' steps are carried out along \textit{constrained partial descent directions} (line \ref{line::theta_2}), found solving the problem
	\begin{equation}
		\label{eq::partial-dir}
		\begin{aligned}
			\theta_{J_{x_c}}^\mathcal{I}(z_c^k) = \min_{d \in \mathbb{R}^n}\;&\max_{j \in \mathcal{I}} \nabla f_j(z_c^k)^\top d + \frac{1}{2}\|d\|^2
			\\\text{s.t. }& A(z_c^k+d) \le b, \quad \mathbf{1}_n^\top(z_c^k+d) = 1 \\& d_{\bar{J}_{x_c}} = \mathbf{0}_{|\bar{J}_{x_c}|}, \quad z_c^k+d \ge \mathbf{0}_n.
		\end{aligned}
	\end{equation}
	We indicate with $v_{J_{x_c}}^\mathcal{I}(z_c^k)$ the optimal solution of the problem; similarly as $v_{J_{x_c}}(x_c)$, the solution $v_{J_{x_c}}^\mathcal{I}(z_c^k)$ is unique; clearly,  when $\mathcal{I} = \{1,\ldots,m\}$, we have $\theta_{J_{x_c}}^\mathcal{I}(z_c^k)=\theta_{J_{x_c}}(z_c^k)$ and  $v_{J_{x_c}}^\mathcal{I}(z_c^k)=v_{J_{x_c}}(z_c^k)$.
	\item Every time a new point is added to the temporary list $\widehat{\mathcal{X}}^k$, a filtering operation is performed to delete all dominated points associated with the same super support set (lines \ref{line::filtering-1}-\ref{line::filtering-2}). 
\end{itemize}

It is important to remark that each point only keeps getting compared with other solutions associated with the same support set $\mathcal{J}$, i.e., lying in the same subspace induced by $\mathcal{J}$. The algorithm thus spans ``in parallel'' different portions of the Pareto front, one for each super support set present in the set $\mathcal{X}$. 

\begin{remark}
	Since $\mathcal{X}^0$ contains feasible points, by definition of constrained descent direction (problems \eqref{eq::common-dir}-\eqref{eq::partial-dir}) and the Armijo-type line searches (lines \ref{line::Armijo_1}-\ref{line::Armijo_2}), we immediately get that Algorithm \ref{alg::SFSD} always produces feasible solutions for problem \eqref{eq::port-prob}. 
\end{remark}

We report here below the convergence properties of \texttt{SFSD}. Readers can find more details and proofs on the theoretical issues related to Algorithm \ref{alg::SFSD} in Appendix \ref{app::descent_methods}. \rev{\label{rev2.4.b}In fact, although the main principles continue to hold, the presence of constraints other than the sparsity one was not considered in the original work of \texttt{SFSD} and, thus, this significant difference have to be suitably taken into account.}

In order to state the convergence result, we need to introduce the set $X_J^k = \{x \mid \exists\, (x, J) \in \mathcal{X}^k\}$, with $J$ being a super support set, and recall the concept of \textit{linked sequence} \cite{liuzzi16}.

\begin{definition}[{\cite[Definition 4.1]{Lapucci2024cardinality}}]
	Let $X^k_J$ be the sequence of sets of nondominated points, associated with the super support set $J$, produced by Algorithm \ref{alg::SFSD}. We define a linked sequence as a sequence $\{x_{j_k}^J\}$ such that, for all $k$, the point $x_{j_k}^J \in X^k_J$ is generated at iteration $k - 1$ of Algorithm \ref{alg::SFSD} while processing point $x_{j_{k-1}}^J \in X^{k-1}_J$.
\end{definition}

\begin{proposition}
	\label{prop::conv-SFSD}
	Let us assume that $X^0_J$ is a set of mutually nondominated points and there exists $x_0^J \in X^0_J$ such that the set $\widehat{\mathcal{L}}_F(x_0^J) = \bigcup_{j=1}^m\{x \in \Omega \mid f_j(x) \le f_j(x_0^J)\}$ is compact. Moreover, let $X^k_J$ be the sequence of sets of nondominated points, associated with the super support set $J$, produced by Algorithm \ref{alg::SFSD}, and $\{x_{j_k}^J\}$ be any linked sequence. Then, $\{x_{j_k}^J\}$ admits accumulation points, each one being Pareto stationary for problem \eqref{eq::port-prob} restricted to the subspace induced by $J$.
\end{proposition}
\begin{proof}
	See Appendix \ref{app::descent_methods}.
\end{proof}

\subsection{Generation of Starting Points}
\label{subsec::phase1}

Algorithm \ref{alg::SFSD} effectively works if solutions with ``good'' supports are provided as input (line \ref{line::initialization}). In this section, we list some of the possible methods that can be used in this preliminary, yet crucial phase.

\begin{description}[font=\textbf]
	\item[Weighted-sum Scalarization Method] A simple choice is to solve multiple times the scalarized problem with different configurations for the weights $\lambda$. However, as already outlined in section \ref{sec::limits_scalarization}, the \rev{linear} scalarization method may miss ``good'' supports and, as a consequence, may not allow \texttt{SFSD} to span large portions of the Pareto front. The following strategies aim to overcome this issue.
	\item[Descent Methods] 
	Following the study in \cite{Lapucci2024cardinality}, gradient based methods can be exploited, running in a multi-start fashion, for this task:
	\begin{itemize}
		\item The first available option is the \textit{Multi-objective Sparse Penalty Decomposition} (\texttt{MOSPD}) approach \cite{Lapucci2022apenalty}, which generates a sequence $\{(x_{k+1}, y_{k+1})\}$ such that, for all $k$, $x_{k+1}$ is approximately Pareto-stationary for the problem
		\begin{align*}
			\min_{x}\;&Q^{\tau_k}(x, y_{k+1}) = F(x) + \mathbf{1}_m\frac{\tau_k}{2}\|x - y_{k+1}\|^2\\\text{s.t. }&Ax\le b, \quad \boldsymbol{1}_n^\top x=1,\quad x\ge \boldsymbol{0}_n,
		\end{align*} with $\tau_k > 0$, whereas $y^{k+1}$ is the projection of $x^{k+1}$ onto the sparse feasible set, i.e., $$y^{k+1} \in \argmin_{y \in \mathbb{R}^n, y \ge \mathbf{0}_n, \|y\|_0 \le s}\frac{1}{2}\|x^{k+1}-y\|^2.$$ \rev{\label{rev1.4.a} As for the first problem, we found in preliminary experiments that the explicit handling of simplex constraints was beneficial w.r.t.\ incorporating them into the penalty term - as we would need to do with more general constraints.}
		Algorithmically, such a pair $(x_{k+1}, y_{k+1})$ is obtained at each iteration by means of an \textit{alternate minimization} scheme. As $\tau_k\to\infty$, $\{x^k\}$ is proved to converge to limit points $\bar{x}$, each one Pareto-stationary for the problem restricted to the subspace induced by one of the super support sets in $\mathcal{J}(\bar{x})$; computationally, the Penalty Decomposition approach was shown to be effective at finding good quality solutions. 
		\item A second approach, called \textit{Multi-Objective Iterative Hard Thresholding} (\texttt{MOIHT}) \cite{Lapucci2024cardinality}, can also be adapted to deal with additional linear constraint and finally generate a solution $x^\star$ such that 
		\begin{equation*}
			\begin{aligned}
				\theta_L(x^\star) = \min_{d \in \mathbb{R}^n}\;&\max_{j \in \{1,\ldots,m\}} \nabla f_j(x^\star)^\top d + \frac{L}{2}\|d\|^2 = 0
				\\\text{s.t. }& A(x^\star+d) \le b, \quad \mathbf{1}_n^\top(x^\star+d) = 1 \\& \|x^\star + d\|_0 \le s, \quad x^\star+d \ge \mathbf{0}_n,
			\end{aligned}
		\end{equation*}
		with $L>0$.
		The above condition, called $L$-stationarity, is necessary for optimality and is stronger \rev{- i.e., implies and is not necessarily implied by -} than the \rev{Pareto-stationarity type condition} achieved by \texttt{MOSPD} and \texttt{SFSD}. 
		\item A hybrid approach is also possible where \texttt{MOIHT} and \texttt{MOSPD}  are run in cascade; the solutions resulting from the two executions are then merged to form the set to give as input to \texttt{SFSD}; doing so, we can exploit both the exploration capabilities of \texttt{MOSPD} and the stronger $L$-stationarity condition to escape from bad Pareto stationary points. More information on the theoretical properties of the two approaches\rev{\label{rev2.4.c}, modified to account for the presence of constraints other than sparsity,} can be found in Appendix \ref{app::descent_methods}.
	\end{itemize}

	\item[Genetic Methods] Genetic approaches, although merely heuristic, have proved to be effective with highly irregular problems. The \textit{Non-dominated Sorting Genetic Algorithm II} (\texttt{NSGA-II}) \cite{deb02} is a famous instance of this class. 
	
	Like any other evolutionary approach, \texttt{NSGA-II} works with a fixed size population of $N$ solutions; the latter is evolved at each iteration through the application of specific random-based operators: \textit{crossover}, \textit{mutation} and	\textit{selection}. In particular, the first two operations aim to generate new and hopefully better solutions from selected ones of the current population; after the addition of the new solutions, the resulting population is then reduced leaving exactly $N$ members. The selection is driven by two metrics, i.e., \textit{ranking} (sorting solutions based on the dominance relation) and \textit{crowding distance} (measuring solution density).
	
	Actually, \texttt{NSGA-II} has been already proposed to solve sparse MOO portfolio problems \cite{Chen2017}: the only fix needed by the base algorithm is to project each point onto the feasible set before being inserting it into the population. \rev{\label{rev1.4.b}The latter operation can in fact be approximated by projecting onto the sparse set and  then dividing the result by its $\ell_1$-norm -- both being computationally cheap steps.}
	
	On the other hand, a memetic extension of \texttt{NSGA-II}, called \texttt{NSMA}, has been recently proposed for global continuous multi-objective optimization \cite{Lapucci2023}: at the end of each \texttt{NSGA-II} iteration, all points in the population are optimized via a descent method. Here we can thus think of adapting \texttt{NSMA} to tackle problem \eqref{eq::port-prob}. Instead of simply projecting the offspring solutions onto the feasible set, we can indeed carry out a sequence of steps of \texttt{MOIHT}: not only we are guaranteed to satisfy the sparsity constraint, but we also improve the quality of the solution, speeding up the whole process.
\end{description}

\smallskip
\noindent The preliminary phase is completed associating each of the solutions found with one of the above methods to one of its super support sets. Finally, $\mathcal{X}^0$ is created filtering out dominated points w.r.t.\ $F$.

\section{Computational Experiments}
\label{sec::experiments}

In this section, we present vast computational experiments aimed at identifying the most effective strategies to tackle problems of the form \eqref{eq::port-prob}. In the following, we are firstly going to analyze the available options for generating initial solutions within \texttt{SFSD}, measuring the variety of ``optimal'' supports retrieved; then, we will show the empirical benefits of using \texttt{SFSD} compared to other baseline approaches; we will finally compare different \texttt{SFSD}-based algorithms with state-of-the-art methodologies, reporting the overall performance in a large collection of problems.

To carry out this detailed analysis, we considered a benchmark of real-world datasets consisting in both daily and monthly prices of a subset of stocks belonging to major financial indices. Similarly as in \cite{Lapucci2022apenalty}, we employed datasets DTS1, DTS2 and DTS3 - obtained from daily prices of the FTSE 100 index from 01/2003 to 12/2007 and containing 12, 24 and 48 securities respectively - and FF10, FF17 and FF48 - extracted from Fama/French benchmark collection, using monthly returns from 07/1971 to 06/2011 and containing 10, 17, 48 securities. In both cases datasets were generated following \cite{BritoVicente2014Efficient, Cocchi2020Concave}, resulting in a set of 6 Mean-Variance portfolio problems.

We furthermore considered four datasets extracted from major financial indexes: EuroStoxx 50, NASDAQ 100, FTSE 100 and S\&P 500. These datasets\footnote{Daily prices and ESG scores are available at \href{https://www.francescocesarone.com/data-sets}{francescocesarone.com/data-sets.}} consist of daily prices, adjusted for dividends and stock splits, and of daily ESG scores obtained from Refinitiv for the time window from 01/01/2013 to 12/31/2020. We assume each index as the benchmark for beta calculations. We selected 50, 54, 80 and 101 stocks from each index respectively. These datasets contain all the parameters required to compute the objective functions mean, variance, ESG score, Sharpe ratio and skewness and also the beta constraints.The expected returns, the covariance matrix and the coskewness tensor were computed from historical data following the sample estimation approach as in \cite{LaiTsong}. The betas for individual stocks were also computed using sample-based methods, assuming the market portfolio returns to be those of the corresponding benchmark index. For the ESG scores, we assumed the last observed value in the respective time series as the score for each stock. 

For each dataset, different instances of the problem were considered, changing the value $s$ of the upper bound on the $\ell_0$ norm. Table \ref{tab::problems} summarizes the details for each dataset employed in the experimentation. For the datasets EuroStoxx, NASDAQ, FTSE and S\&P we considered the \rev{following} combinations of objectives\rev{:}
\rev{\label{rev1.2.a}\begin{itemize}
    \item $m=2$: Mean-Variance, SR-ESG, SR-Skew, ESG-Skew;
    \item $m=3$: SR-ESG-Skew, Variance-Mean-ESG, Variance-Mean-Skew
    \item $m=4$: Variance-Mean-ESG-Skew. 
\end{itemize}}
For each combination we consider the problem with and without beta constraints (in the former case $\beta_{\text{min}}=0.8$ and $\beta_{\text{max}} = 1.2$). 
In total, the benchmark  consists of $442$ problems ($210$ with $s \le 10$ and $232$ with $s>10$). 
To obtain values at similar scales for all the objective functions, we multiplied the objectives by a $10^2$ factor for both variance and mean, $10^{-2}$ for ESG score, $10^{-1}$ for skewness, while no scaling was applied to Sharpe ratio. 

\begin{table*}
	\centering
	\footnotesize
	\caption{Details of the datasets employed in the experimental analysis.}
	\label{tab::problems}
	\renewcommand{\arraystretch}{1.4}
	\begin{tabular}{|c||c|c|}
		\hline
		\multirow{2}{*}{\textbf{Dataset}} & \multirow{2}{*}{$\mathbf{n}$} & \multirow{2}{*}{$\mathbf{s}$} \\
		&&\\
		\hline \hline
		DTS1 & 12 & 2,5,10 \\ \hline
		DTS2 & 24 & 2,5,10,15 \\ \hline
		DTS3 & 48 & 2,5,10,15,20,30 \\ \hline
		FF10 & 10 & 2,5,10 \\ \hline
		FF17 & 17 & 2,5,10,15 \\ \hline
		FF48 & 48 & 2,5,10,15,20,30 \\ \hline
		EuroStoxx & 50 & 2,5,10,15,20,30  \\ \hline
		NASDAQ & 54 & 2,5,10,15,20,30 \\ \hline
		FTSE & 80 & 2,5,10,15,20,30,50 \\ \hline
		S\&P & 101 & 2,5,10,15,20,30,50  \\ \hline
		
		\hline
	\end{tabular}
\end{table*}

In all the experiments we set $10^{-7}$ as tolerance for nonzero components in the evaluation of the zero norm sparsity. In \texttt{MOIHT}, $L$ is selected as $1.1$ times the Lipschitz constant of $f^V$, i.e., $L=1.1\lambda_{\text{max}}(Q)$, with $\lambda_{\text{max}}(Q)$ being the largest eigenvalue of matrix $Q$; we set $\theta_{tol}=-10^{-7}$ as a threshold to discriminate stationary points. In \texttt{MOSPD}, we employ $\|x_{k+1} - y_{k+1}\| \le 10^{-3}$ as stopping condition, $\tau_0=10^{-2}$, $\tau_{k+1}=2 \tau_k$, $\varepsilon_0=10^{-3}$ and $\varepsilon_{k+1} = 0.9 \varepsilon_k$. For \texttt{NSGA-II} and \texttt{NSMA} we set the maximal population size to 100, while standard settings for crossover and mutation probabilities are used. Single-objective problems resulting from \rev{linear} scalarization are solved (to global optimality) using Gurobi v11 \cite{gurobi} for problems with linear and quadratic objectives, whereas the penalty decomposition scheme \cite{lapucci2021convergent} is used in the other cases. 
Gurobi v11 is also used to solve all subproblems associated with the computation of descent directions.

For \texttt{SFSD} we set $\theta_{tol} = -10^{-7}$; in order to improve the efficiency of \texttt{SFSD}, points are processed by first considering those that are not dominated. Moreover, a condition based on the crowding distance \cite{deb02} determines the points that are processed within a fixed support, so as to prevent the production of points too close to each other. 
For Armijo-type line-searches we set $\delta=0.5$ and $\gamma=10^{-4}$.  

The initial solutions for \texttt{MOHyb}, \texttt{NSGA-II} and \texttt{NSMA} consist of $2n$ feasible portfolios, where $n$ is the dimensionality of the problem considered. The first $n$ are selected as the components of the standard basis in $\mathbb{R}^n$, the remaining ones are random sparse vectors in $\mathbb{R}^n$. Projection of these points onto the feasible set is then performed.
In \texttt{MOHyb} both \texttt{MOIHT} and \texttt{MOSPD} are fed with the same starting solutions as above. As for the \rev{linear} scalarization method, $2n$ choices of $\lambda$ are considered in order to span as evenly as possible the combinations of the $m$ objectives, while a random feasible portfolio is selected as starting point when needed. Given the randomness in the initialization, we ran each experiment 5 times with different seeds.

The code\footnote{The code is publicly available at \href{https://github.com/dadoPuccio/MO-Portfolio}{github.com/dadoPuccio/MO-Portfolio}.} of the computational experiments reported in this section was implemented in \texttt{Python3} and executed on a computer with an Intel Xeon Processor E5-2430 v2 6 cores 2.50 GHz and 32 GB RAM running Ubuntu 22.04.

\subsection{Analysis of Supports at Initialization}
\label{subsec::phase_1}

As a starting point of our experimental analysis, we need to identify the most effective way to provide \texttt{SFSD} with as many supports associated with Pareto solutions as possible. To accomplish this task, we approximated the true Pareto front of each problem by merging all the front reconstructions obtained by running, from $5$ different seeds ($4$ minutes for each seed and algorithm), \texttt{MOHyb}, \texttt{NSGA-II}, \texttt{NSMA} and the \rev{linear} scalarization method, all followed by \texttt{SFSD} (for $2$ additional minutes). To improve the accuracy of the reference front, we did 5 additional runs of pure \texttt{NSGA-II} ($6$ minutes per run).

Figure \ref{fig::reference} shows the reference fronts obtained for two bi-objective problems: the mean-variance problem on DTS2 and the SR-Skew problem on EuroStoxx. We can observe that, as the value of the $\ell_0$ threshold $s$ increases, the number of optimal supports found grows. This behavior is somewhat natural, as larger values of $s$ lead to an increasing number of possible supports.

\begin{figure}[htbp]
	\centering
	\subfloat[Reference front of DTS2 mean-variance problem - $s=2$.]{\includegraphics[width=0.32\textwidth]{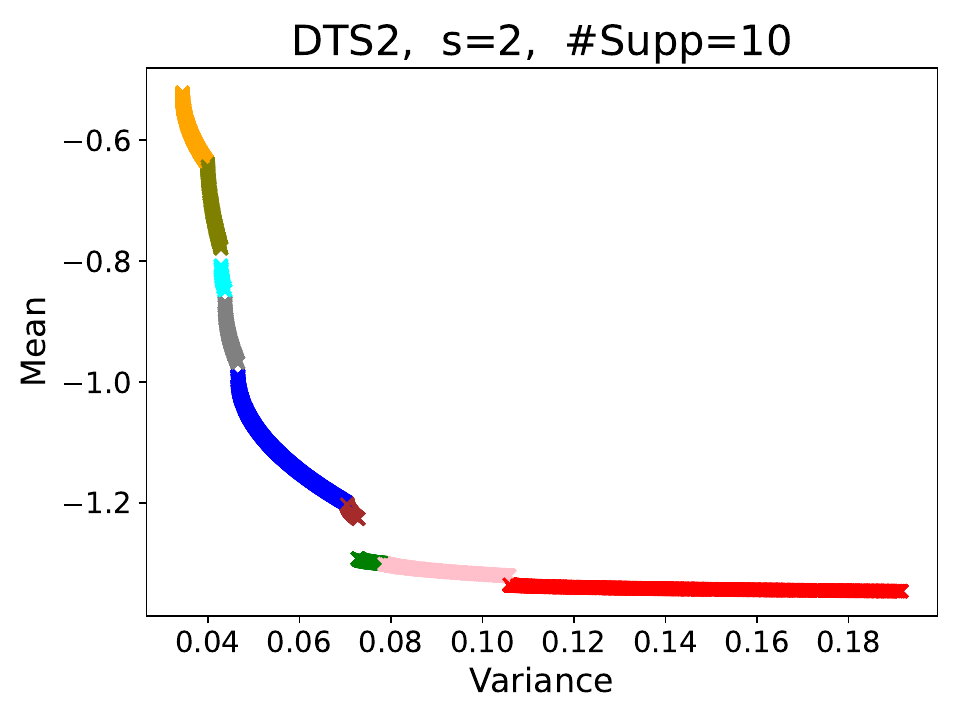}}
	\hfill
	\subfloat[Reference front of DTS2 mean-variance problem - $s=5$.]{\includegraphics[width=0.32\textwidth]{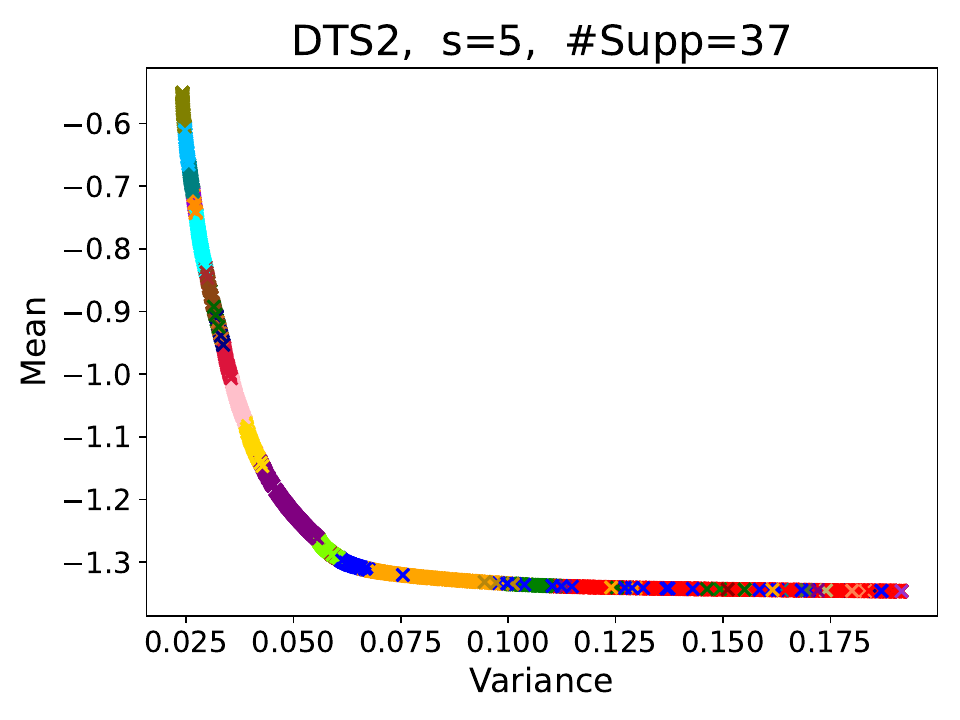}}
	\hfill
	\subfloat[Reference front of DTS2 mean-variance problem - $s=10$.]{\includegraphics[width=0.32\textwidth]{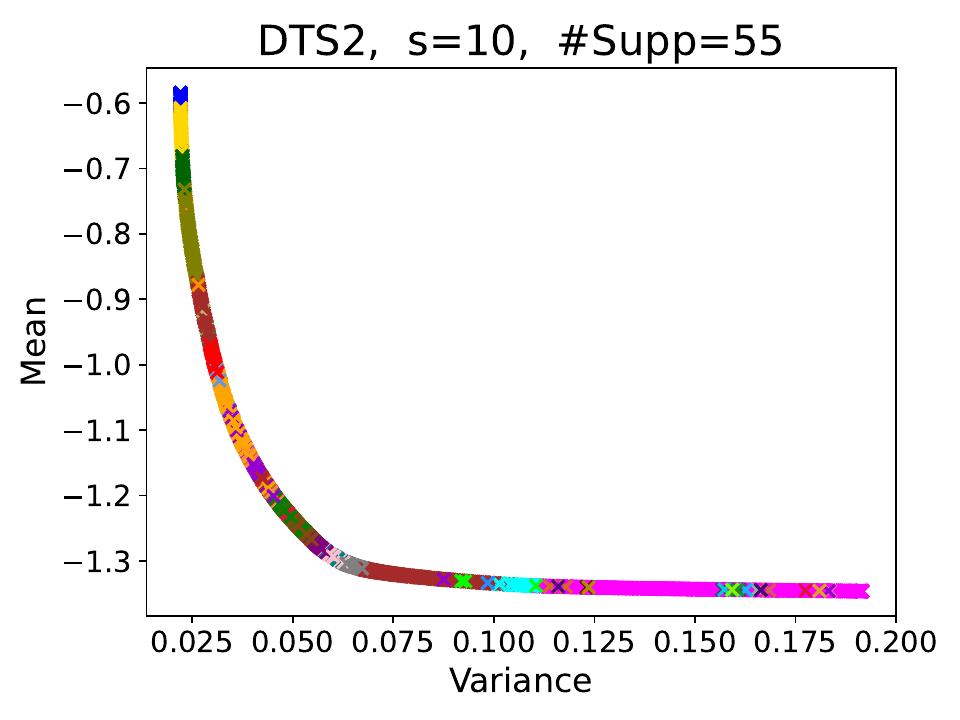}}
	\\
	\subfloat[Reference front of EuroStoxx SR-Skew problem - $s=2$.]{\includegraphics[width=0.32\textwidth]{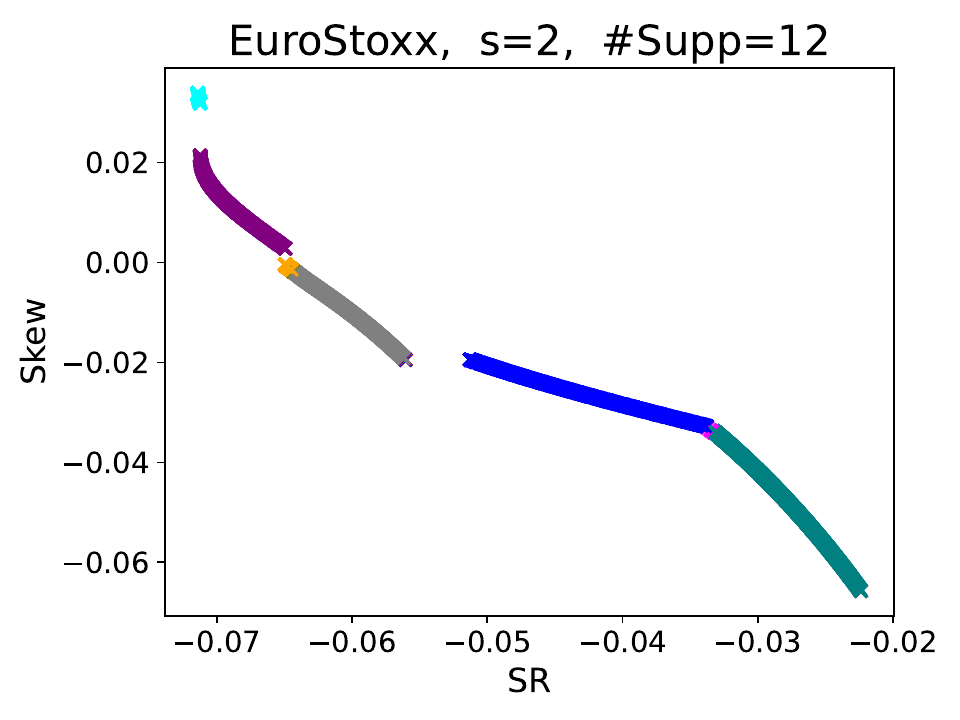}}
	\hfill
	\subfloat[Reference front of EuroStoxx SR-Skew problem - $s=5$.]{\includegraphics[width=0.32\textwidth]{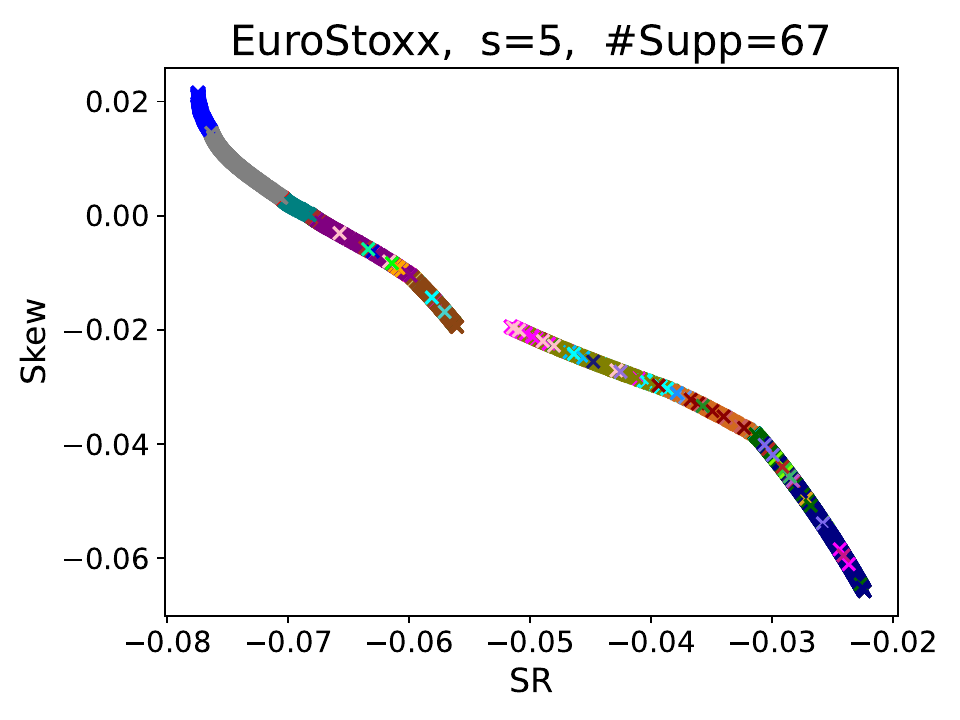}}
	\hfill
	\subfloat[Reference front of EuroStoxx SR-Skew problem - $s=10$.]{\includegraphics[width=0.32\textwidth]{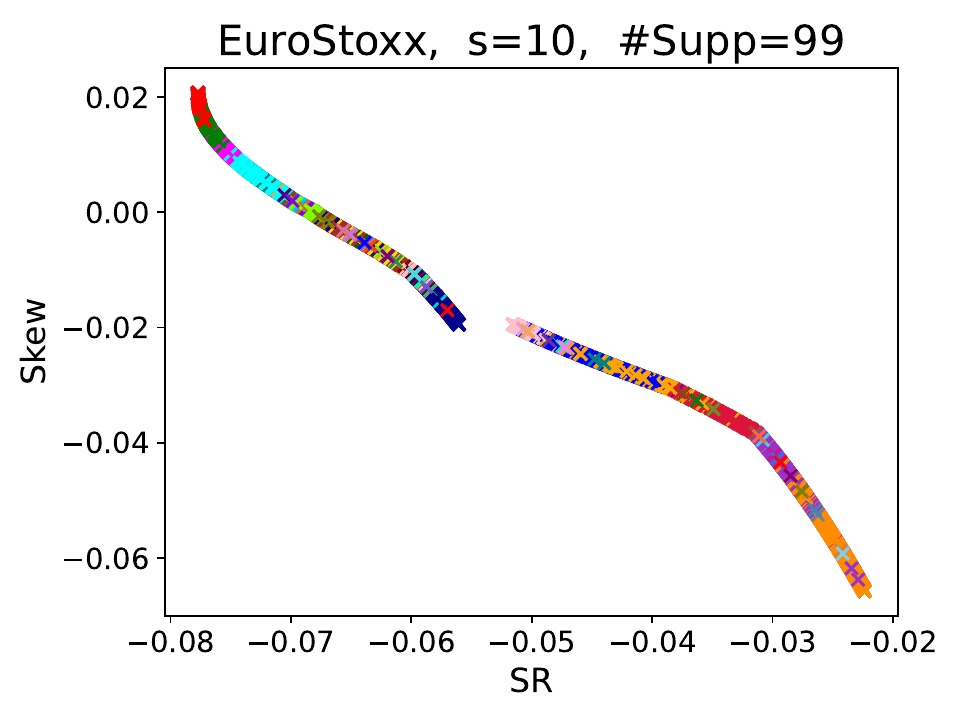}}
	\caption{Reference fronts of selected problems. Different colors denote different supports of the corresponding solutions. The title of each plot reports also the total number of optimal supports in the reference front.}
	\label{fig::reference}
\end{figure}

In order to evaluate each initialization strategy, we measure, for each problem, the ratio between the number of different optimal supports found over the total number of optimal supports in the reference front. Given $S^{(p)}_{m}$ the set of supports of the solutions found in problem $p$ by the method $m$, and $S^{(p)}_R$ the set of supports of the solutions in the reference front of problem $p$, we compute the \textit{recall} of method $m$ in problem $p$ as
$R^{(p)}_{m} = \frac{|S^{(p)}_m \cap S^{(p)}_R|}{|S^{(p)}_R|}$. We evaluated \texttt{MOHyb}, \texttt{NSGA-II}, \texttt{NSMA} and the \rev{linear} scalarization method according to this metric, running them on the entire benchmark of 442 problems \rev{\label{rev1.2.b}involving two or more objective functions}. All methods share a $4$ minutes time limit as a stopping condition. 

\begin{figure}[htbp]
	\centering
	\subfloat[Plot of the recall cumulative in problems with $s\le10$ without beta constraints.]{\includegraphics[width=0.4\textwidth]{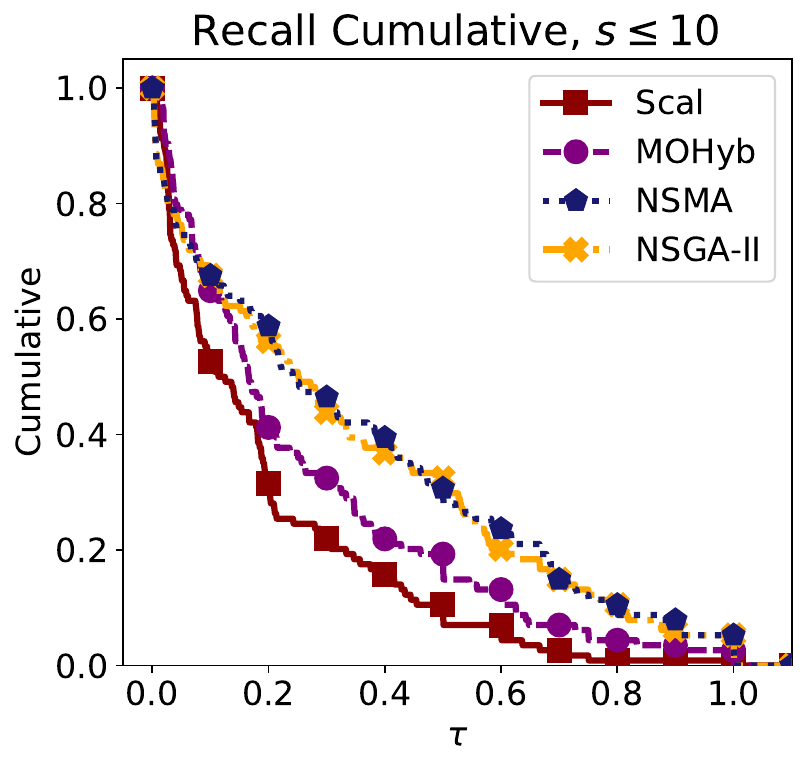}}
	\hfil
	\subfloat[Plot of the recall cumulative in problems with $s>10$ without beta constraints.]{\includegraphics[width=0.4\textwidth]{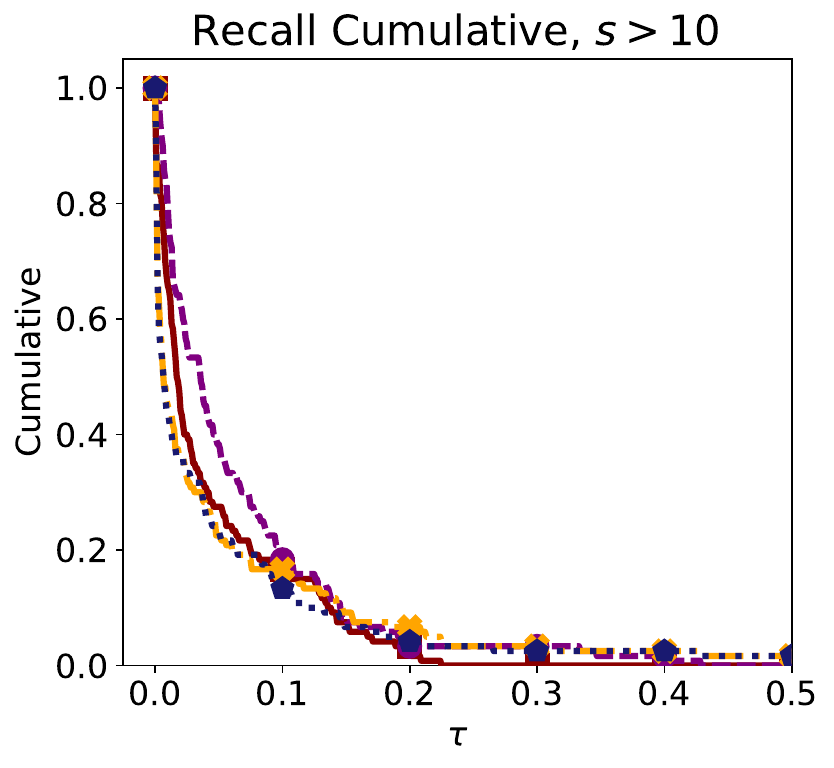}}
	\hfil
	\\
	\subfloat[Plot of the recall cumulative in problems with $s\le10$ with beta constraints.]{\includegraphics[width=0.4\textwidth]{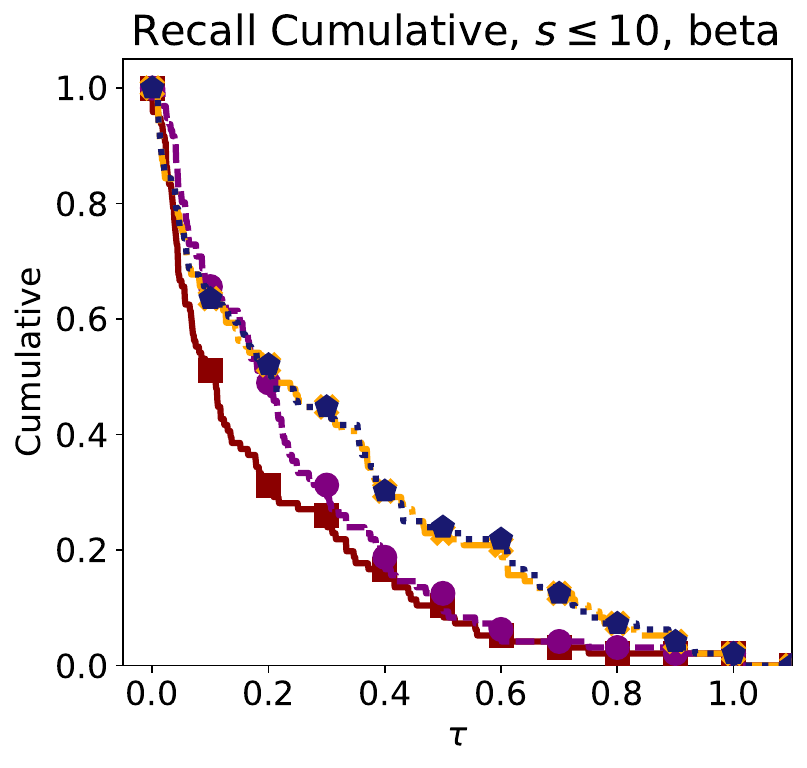}}
	\hfil
	\subfloat[Plot of the recall cumulative in problems with $s>10$ with beta constraints.]{\includegraphics[width=0.4\textwidth]{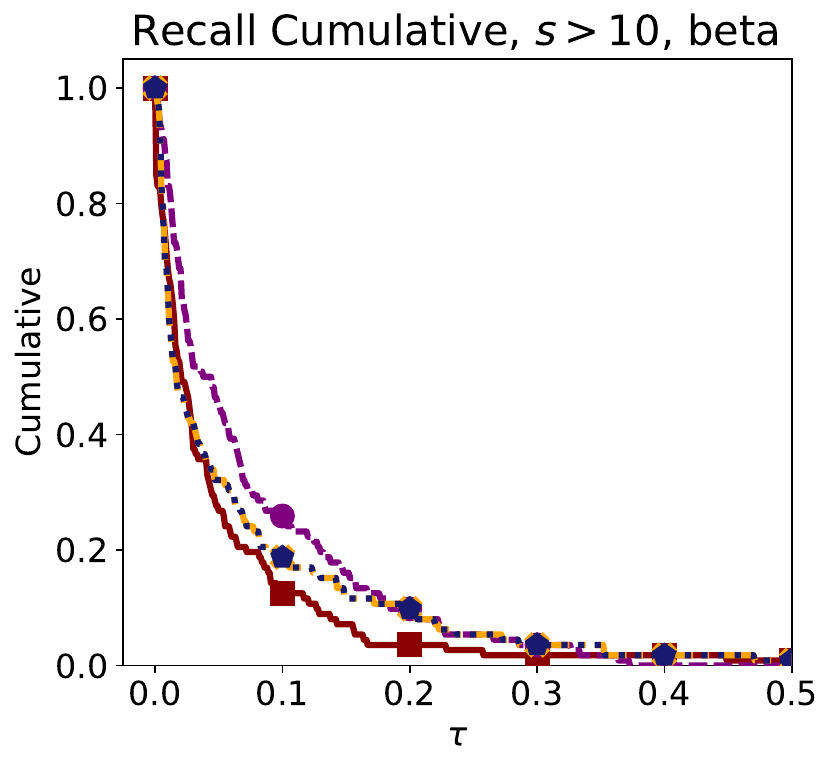}}
	\caption{Plots of the cumulative, across problems \rev{involving two or more objective functions}, of the ratio of optimal supports found by \texttt{MOHyb}, \texttt{NSGA-II}, \texttt{NSMA} and the \rev{linear} scalarization method, relative to the optimal supports of the reference front.}
	\label{fig::opt_supp}
\end{figure}

Given the large number of problems considered, we summarize the outcome of this analysis by plotting the cumulative distribution of the recall values across all problems. The plots in Figure \ref{fig::opt_supp} present in the x-axis the recall threshold $\tau$, while the y-axis indicates the percentage of problems for which the recall is at least $\tau$. We split the analysis between small values of maximal cardinality ($s \le 10$) and large maximal cardinality ($s > 10$), and also between problems with and without beta constraints. We can observe that for smaller values of $s$ the most effective methods are  \texttt{NSMA} and \texttt{NSGA-II}, as their cumulative curve is consistently above other methods, meaning that a larger portion of optimal supports is generally retrieved. For greater values of $s$, \texttt{MOHyb} instead appears to be the most effective one, as the cumulative curve for this method is slightly above the others, especially for smaller values of $\tau$. No significant change is observed when considering problems with the beta constraints.

\subsection{Impact of SFSD-based exploration}
\label{sec:sfsd_impact}

To demonstrate the benefits provided by the employment of \texttt{SFSD}, we compare the Pareto front reconstructions of \texttt{MOHyb}, \texttt{NSGA-II}, \texttt{NSMA} and the \rev{linear} scalarization method, when \texttt{SFSD} is employed on top of those, against the case no further processing of points is performed after the first phase. Each ``phase 1'' method is run, as usual, for 4 minutes, then \texttt{SFSD} is used with a time limit of 2 minutes. Each algorithm is executed 5 times from different initialization, then the solutions of the run that obtain maximal hypervolume \cite{zitzler98} are selected as the best front reconstruction for that method. Figures \ref{fig::SFSD_comp_DTS}-\ref{fig::SFSD_comp_Euro} report the best fronts for each algorithm when \texttt{SFSD} is used and not, on the DTS2 mean-variance problem ($s=2$) and the EuroStoxx SR-Skew problem ($s=5$). We observe that \texttt{SFSD} allows to fill the gaps in the Pareto front, offering a wider set of solutions for the problems considered. We can also note that using \texttt{SFSD} on top of \texttt{NSGA-II} has also the additional effect of driving solutions towards stationarity, since derivatives are used only in this second phase. 

\begin{figure}[htbp]
	\centering
	\subfloat{\includegraphics[width=0.39\textwidth]{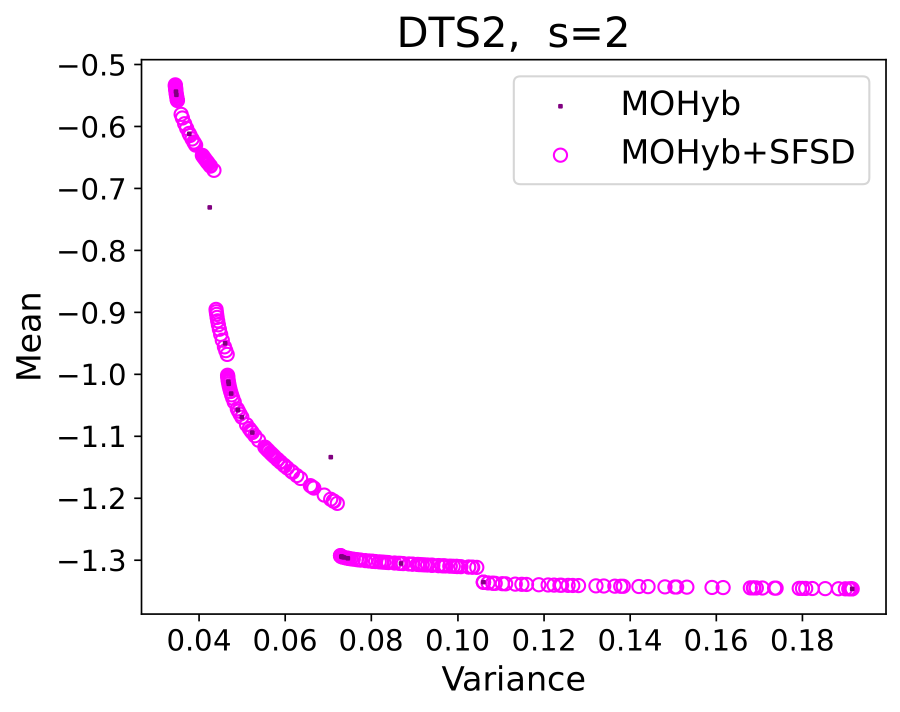}}
	\hfil
	\subfloat{\includegraphics[width=0.39\textwidth]{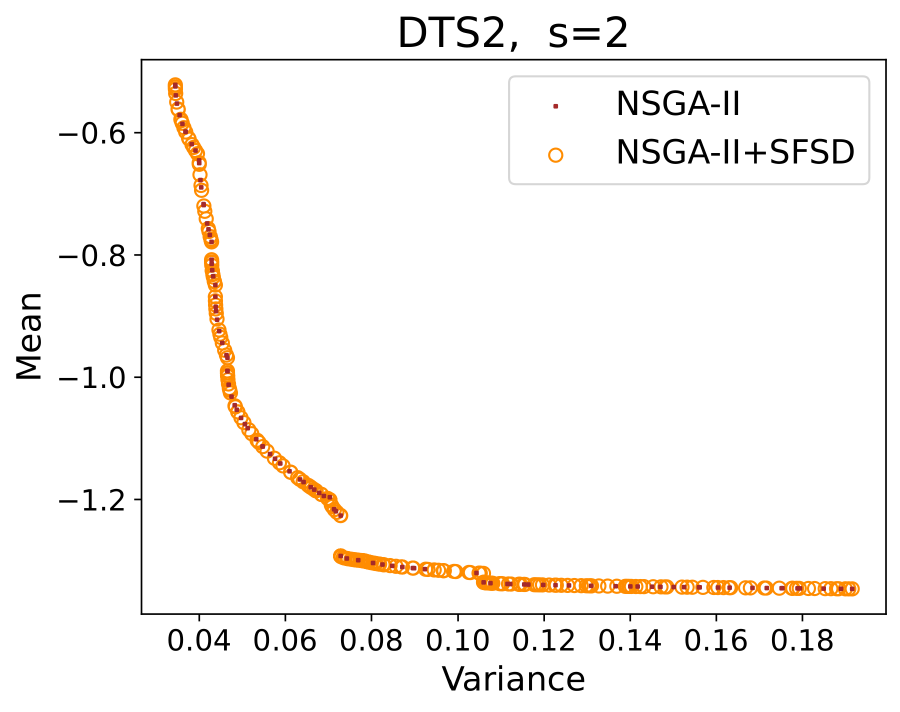}}
	\\
	\subfloat{\includegraphics[width=0.39\textwidth]{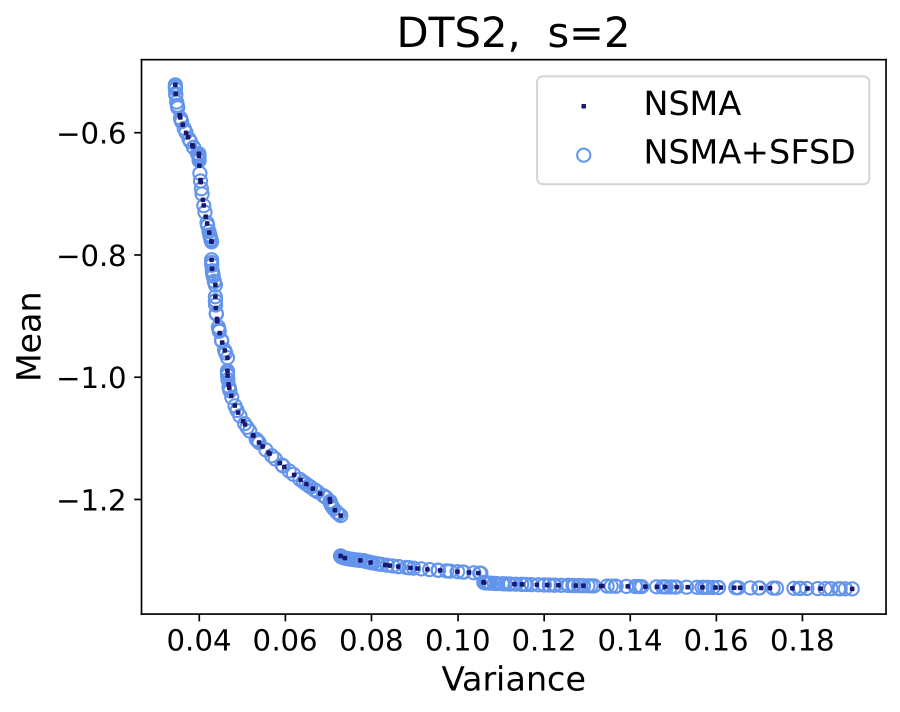}}
	\hfil
	\subfloat{\includegraphics[width=0.39\textwidth]{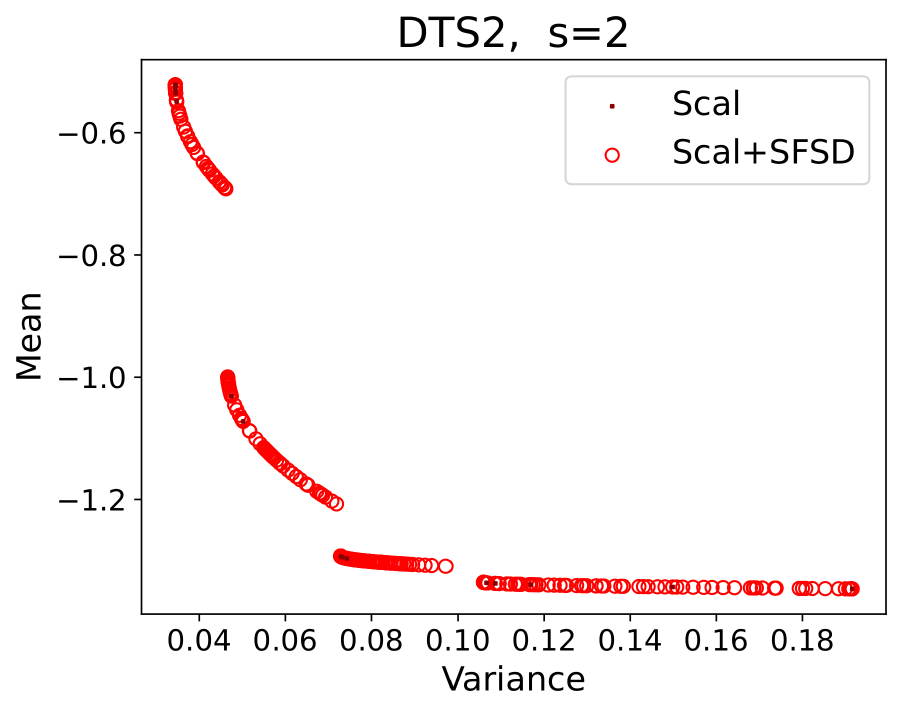}}
	\caption{Comparison of the Pareto front reconstructions of \texttt{MOHyb}, \texttt{NSGA-II}, \texttt{NSMA} and the \rev{linear} scalarization method, when \texttt{SFSD} is used, on the DTS2 mean-variance problem in opposition to the scenario where no further processing of solutions is performed after ``phase 1''.}
	\label{fig::SFSD_comp_DTS}
\end{figure}

\begin{figure}[htbp]
	\centering
	\subfloat{\includegraphics[width=0.41\textwidth]{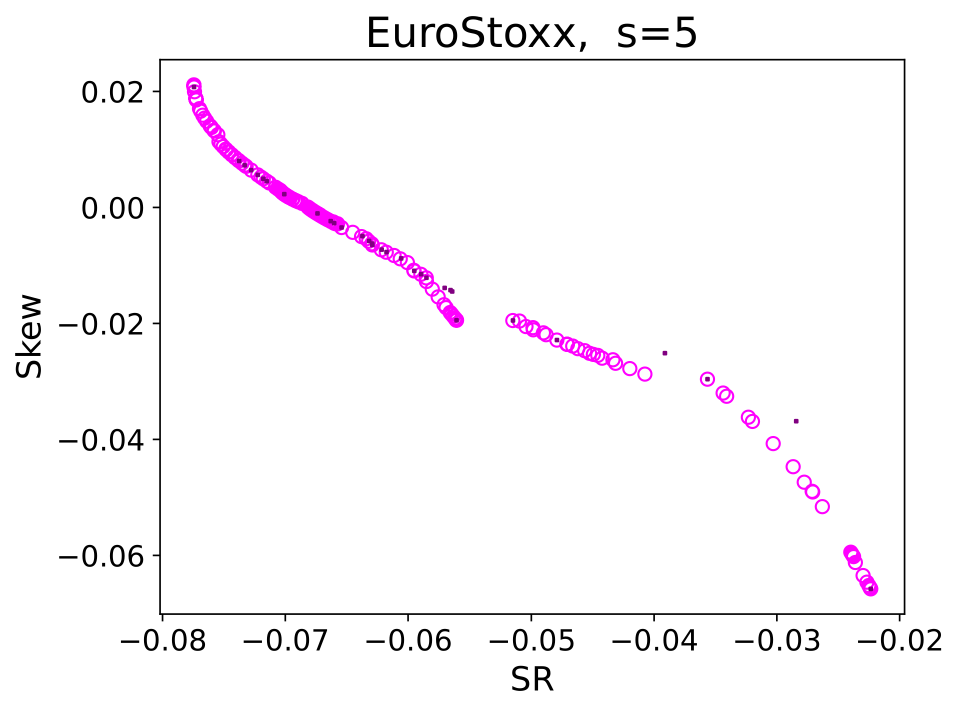}}
	\hfil
	\subfloat{\includegraphics[width=0.41\textwidth]{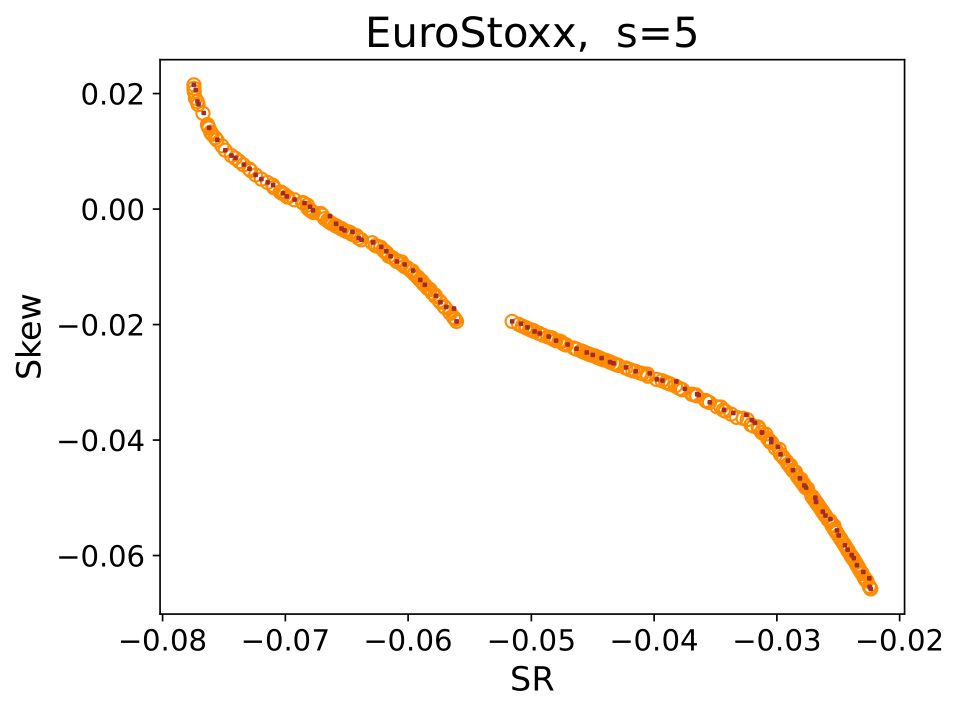}}
	\\
	\subfloat{\includegraphics[width=0.41\textwidth]{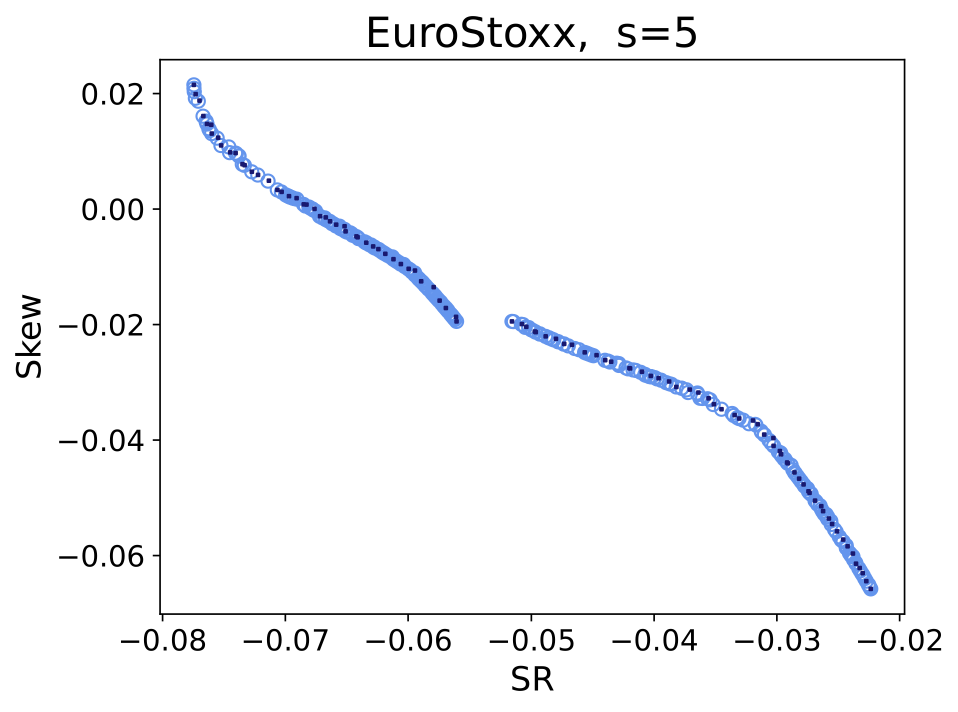}}
	\hfil
	\subfloat{\includegraphics[width=0.41\textwidth]{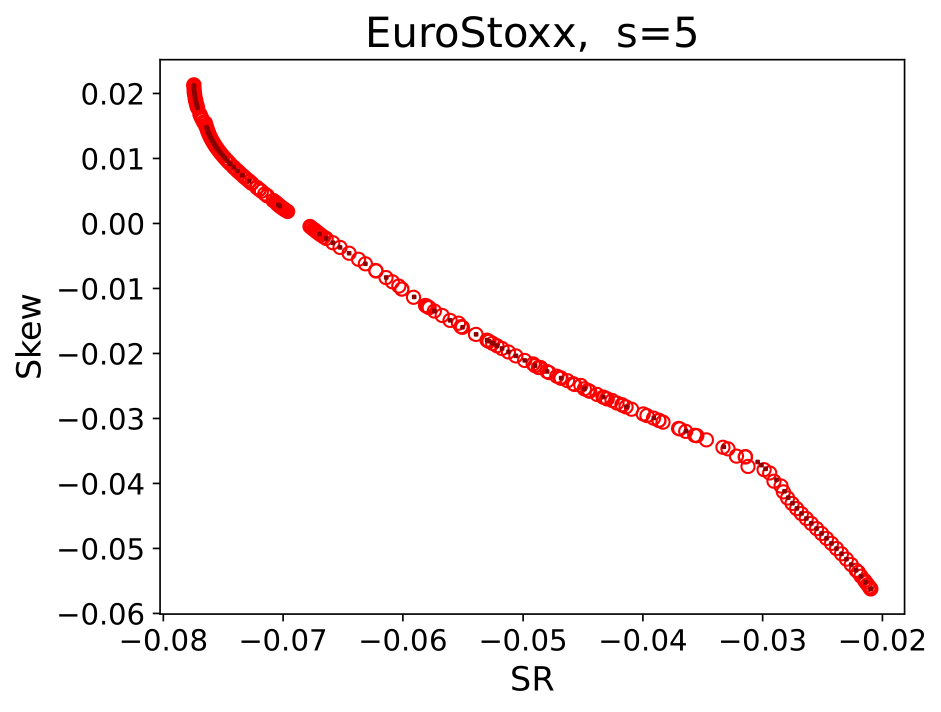}}
	\caption{Comparison of the Pareto front reconstructions of \texttt{MOHyb}, \texttt{NSGA-II}, \texttt{NSMA} and the \rev{linear} scalarization method, when \texttt{SFSD} is used, on the EuroStoxx SR-Skew problem in opposition to the scenario where no further processing of solutions is performed after ``phase 1''.}
	\label{fig::SFSD_comp_Euro}
\end{figure}

We shall observe that the solutions found by \texttt{SFSD} highly depend on the initialization points provided by ``phase 1'' methods. To better investigate on this aspect, we report in Figure \ref{fig::SFSD_comp_all} a comparison of the Pareto front reconstructions when using \texttt{SFSD} on top of \texttt{MOHyb}, \texttt{NSGA-II}, \texttt{NSMA} and the \rev{linear} scalarization method, in a subset of convex and non convex problems. For the sake of completeness, we also report in the plots the ``best'' Pareto front reconstruction of \texttt{NSGA-II} executed with a time limit of 6 minutes, without \texttt{SFSD}. 

\begin{figure}[htbp]
	\centering
	\subfloat{\includegraphics[width=0.33\textwidth]{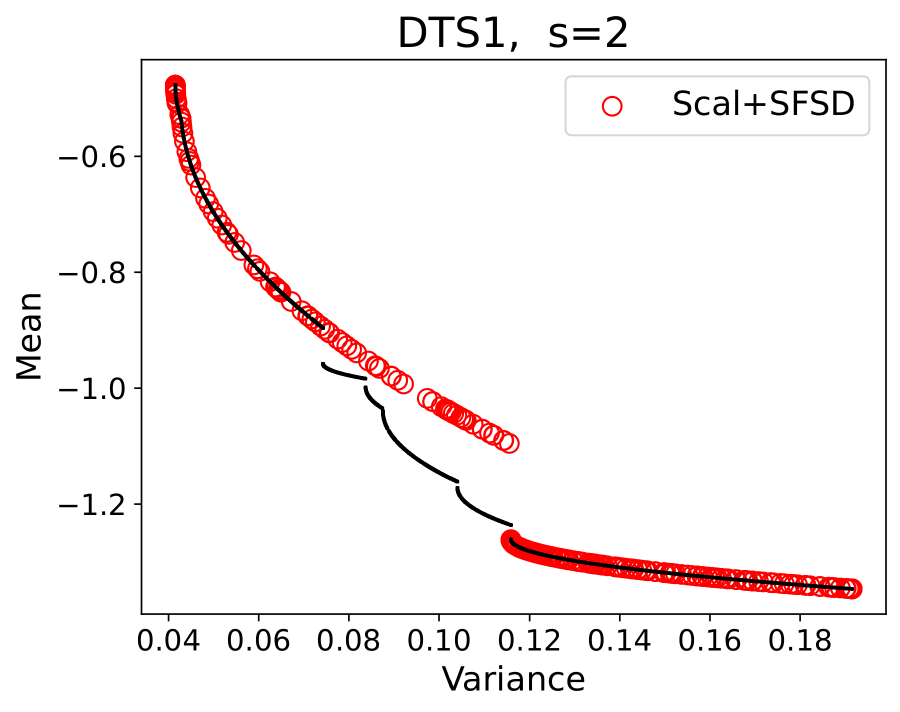}}
	\hfil
	\subfloat{\includegraphics[width=0.33\textwidth]{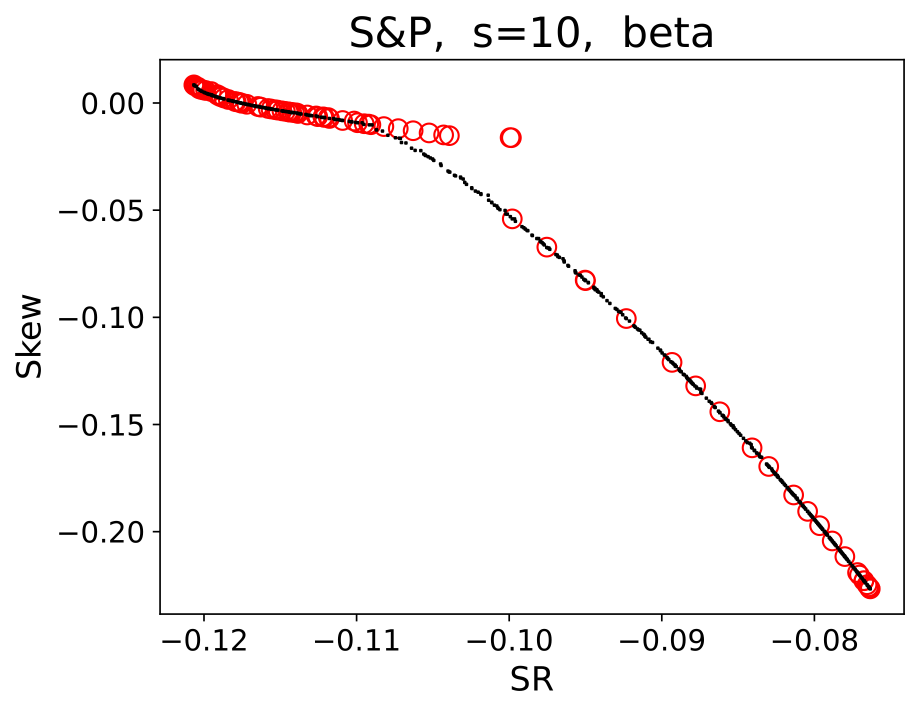}}
	\hfil
	\subfloat{\includegraphics[width=0.33\textwidth]{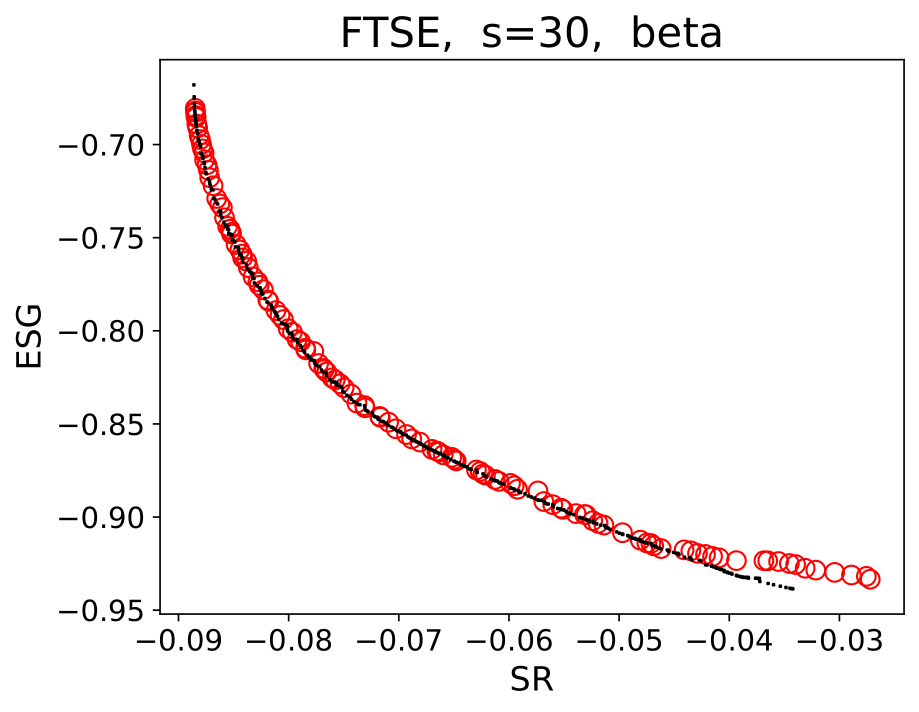}}
	\\
	\subfloat{\includegraphics[width=0.33\textwidth]{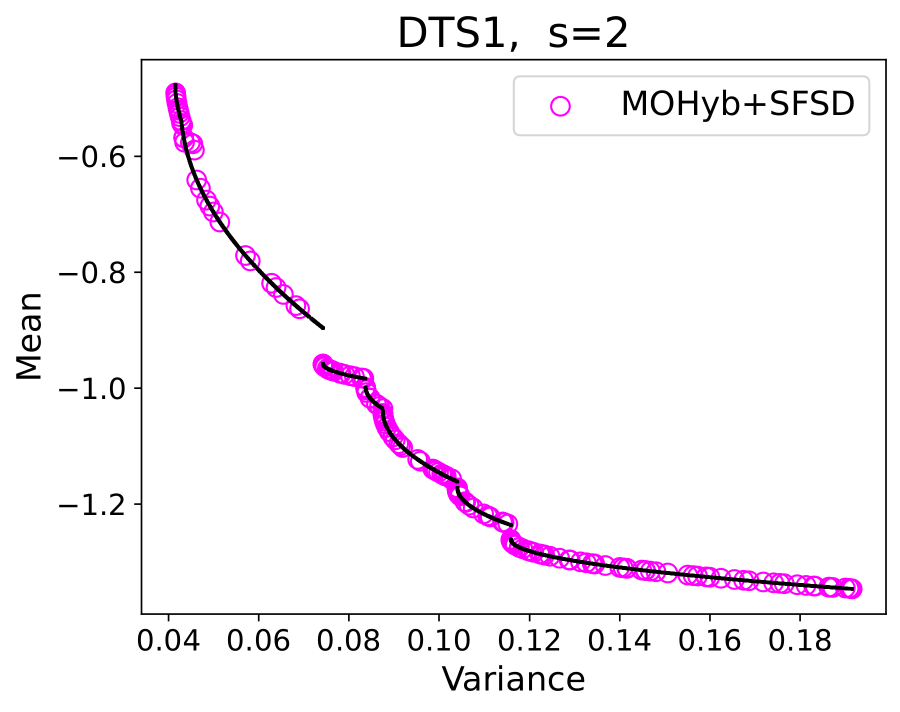}}
	\hfil
	\subfloat{\includegraphics[width=0.33\textwidth]{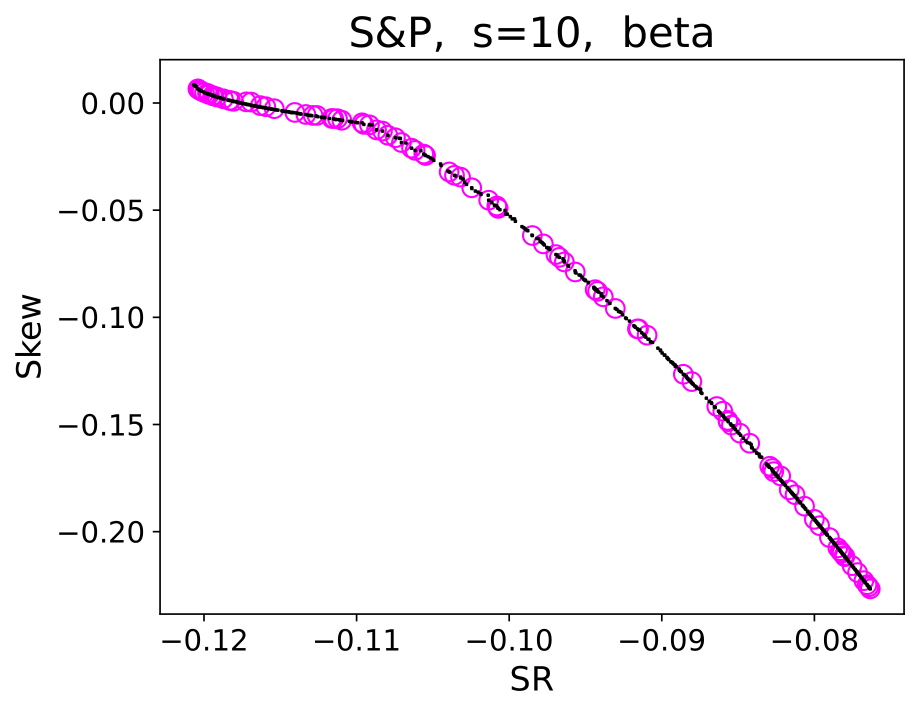}}
	\hfil
	\subfloat{\includegraphics[width=0.33\textwidth]{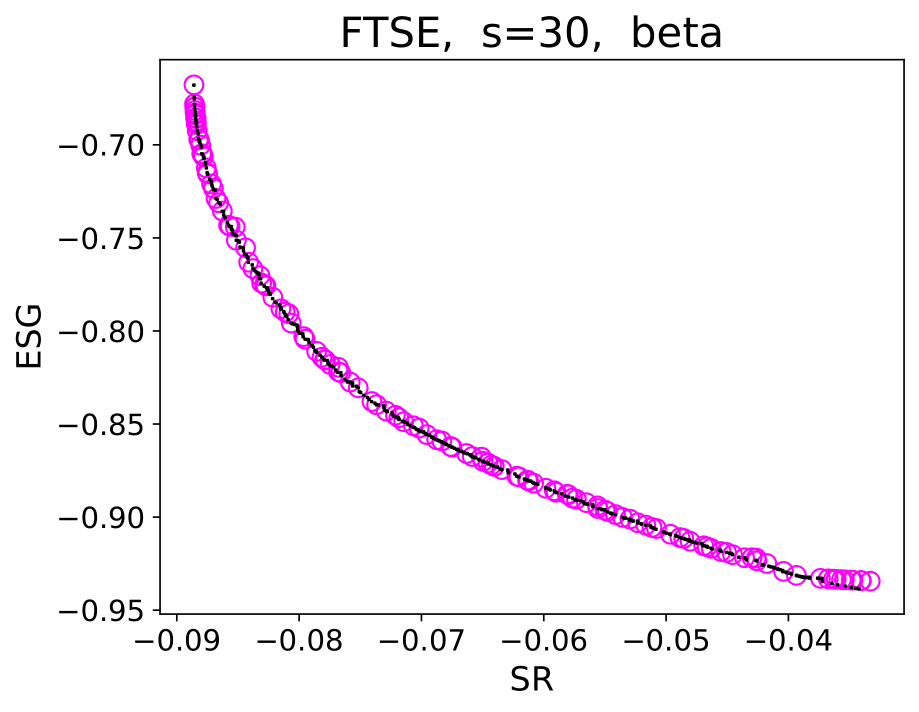}}
	\\
	\subfloat{\includegraphics[width=0.33\textwidth]{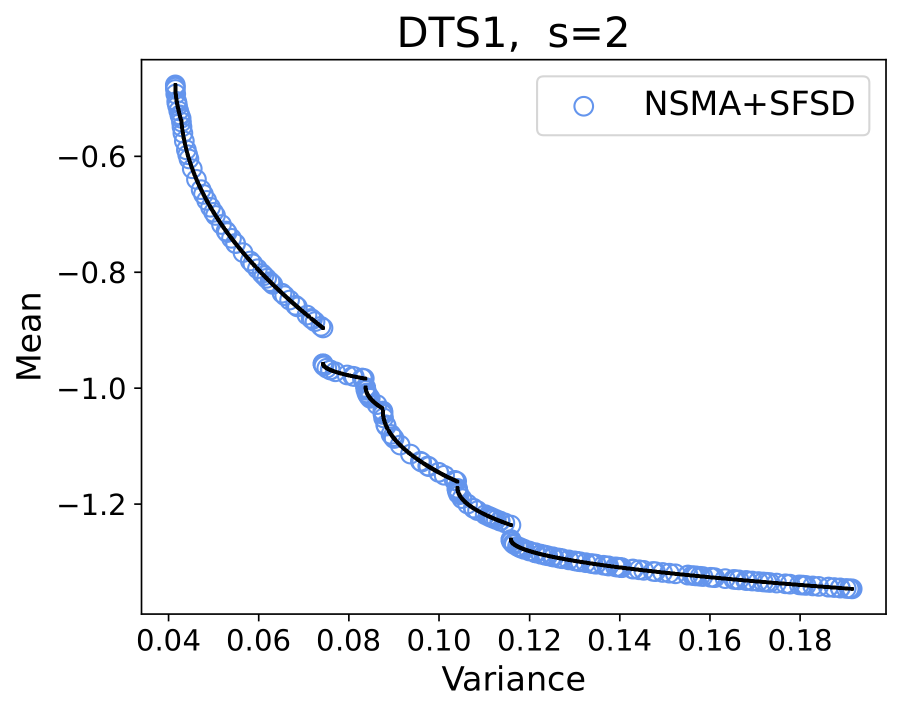}}
	\hfil
	\subfloat{\includegraphics[width=0.33\textwidth]{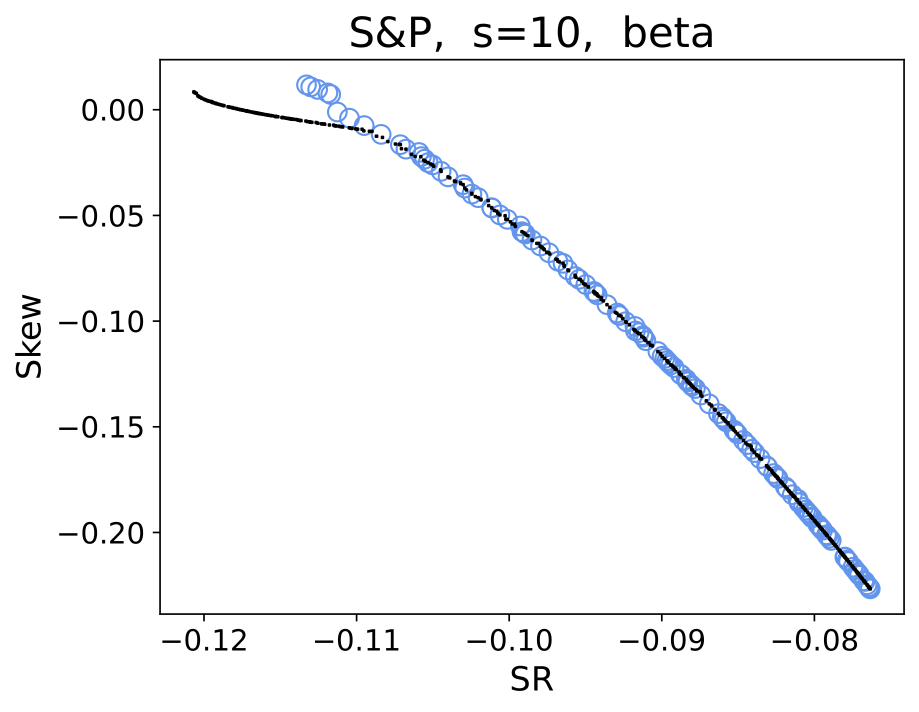}}
	\hfil
	\subfloat{\includegraphics[width=0.33\textwidth]{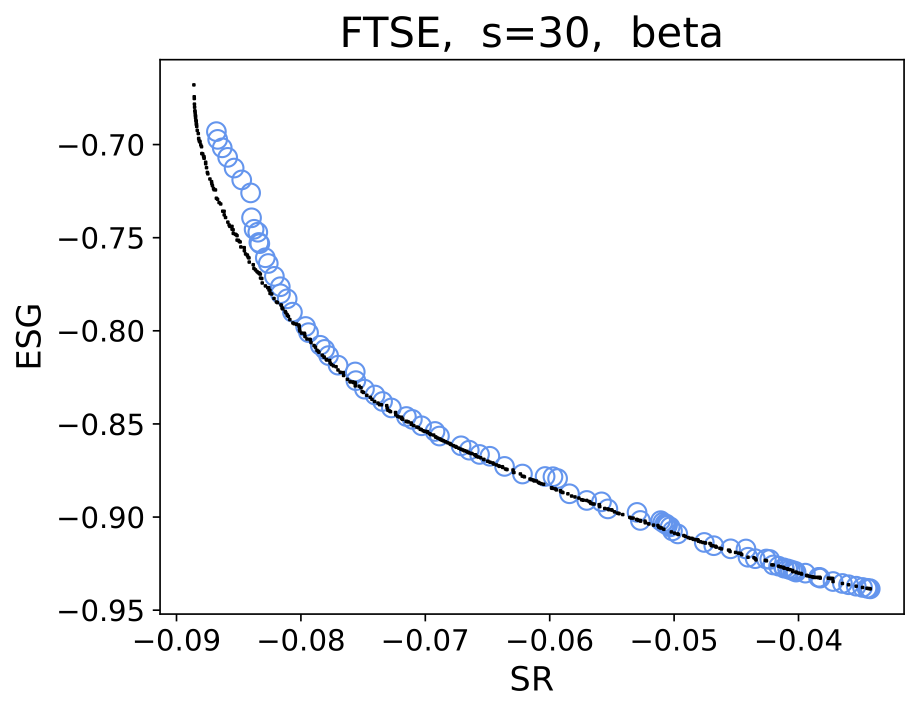}}
	\\
	\subfloat{\includegraphics[width=0.33\textwidth]{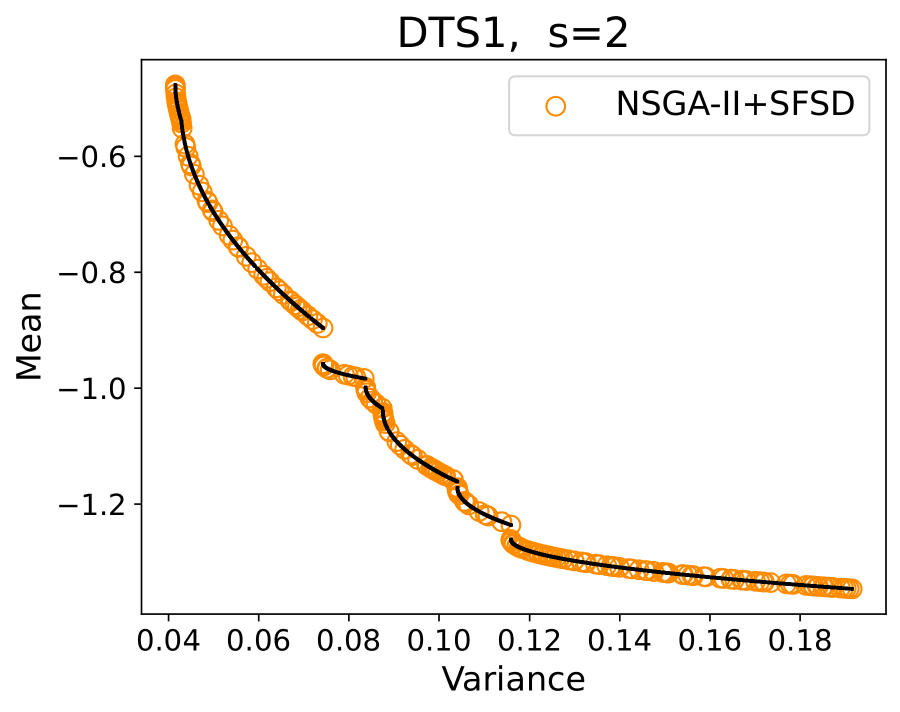}}
	\hfil
	\subfloat{\includegraphics[width=0.33\textwidth]{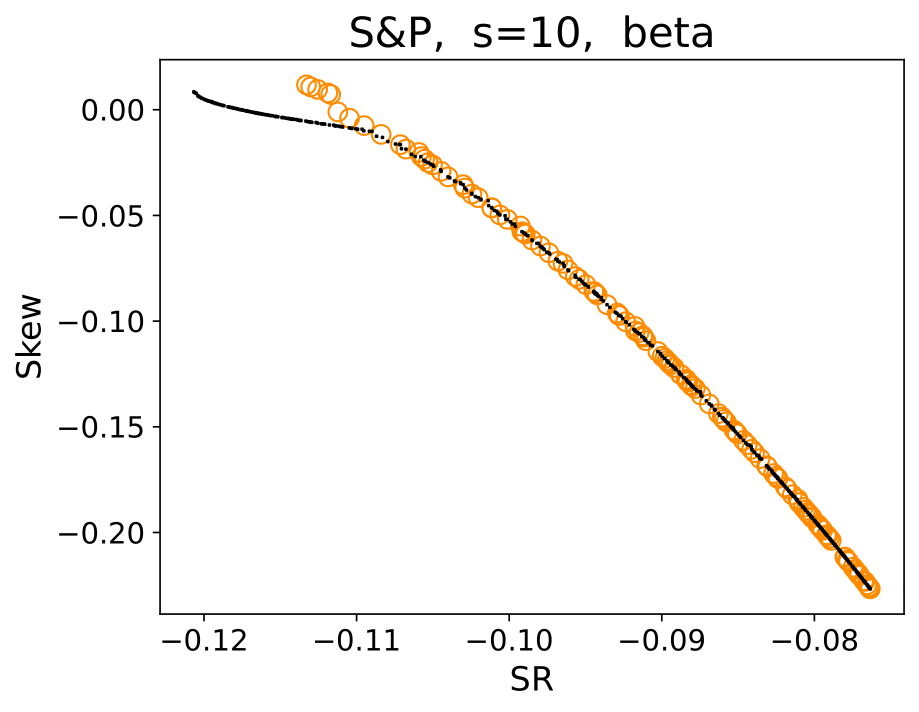}}
	\hfil
	\subfloat{\includegraphics[width=0.33\textwidth]{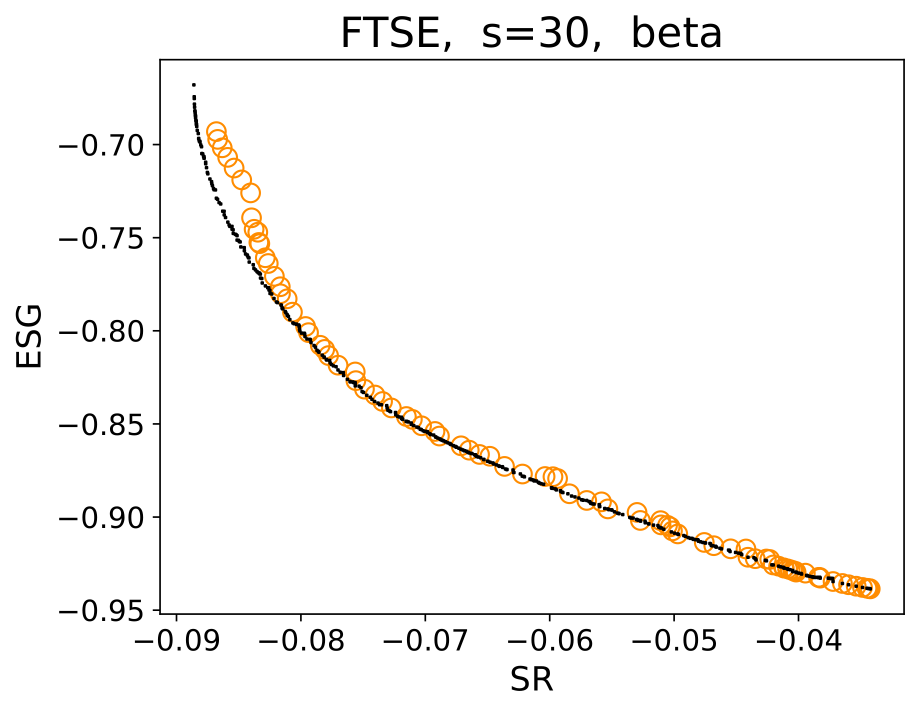}}
	\\
	\subfloat{\includegraphics[width=0.33\textwidth]{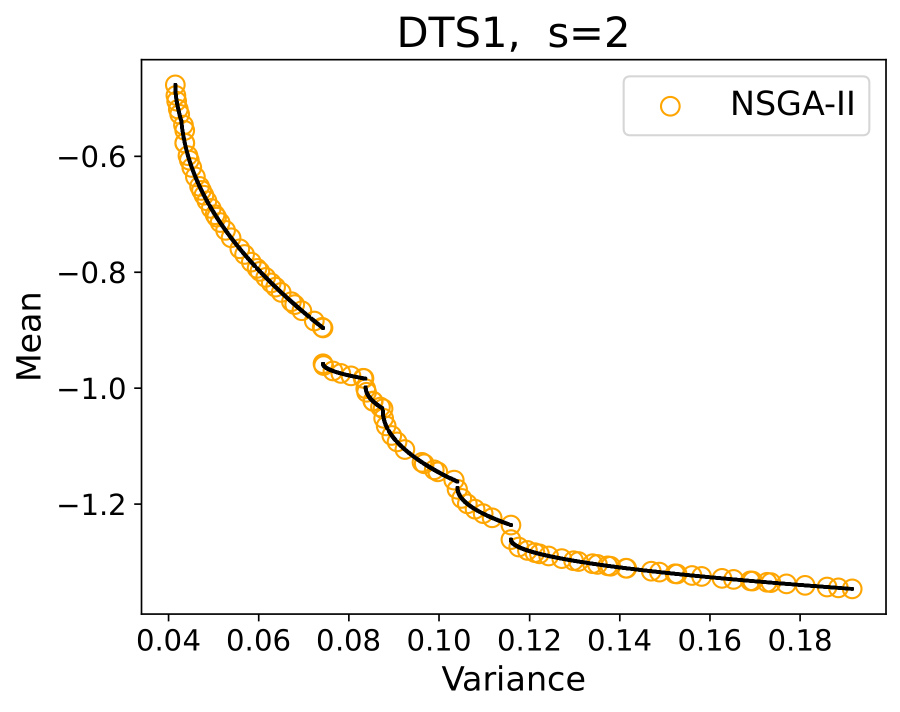}}
	\hfil
	\subfloat{\includegraphics[width=0.33\textwidth]{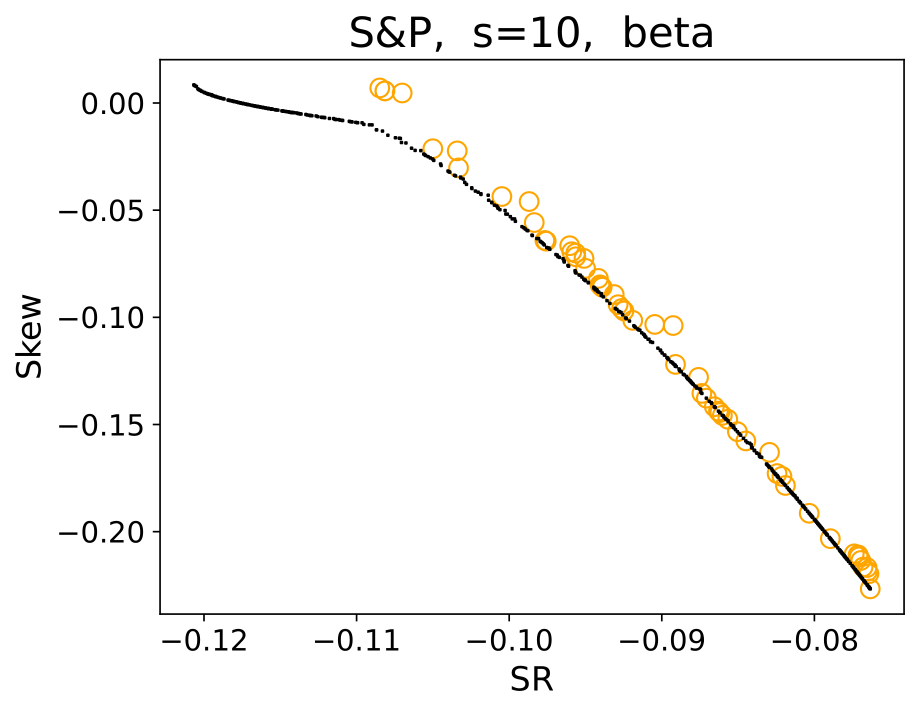}}
	\hfil
	\subfloat{\includegraphics[width=0.33\textwidth]{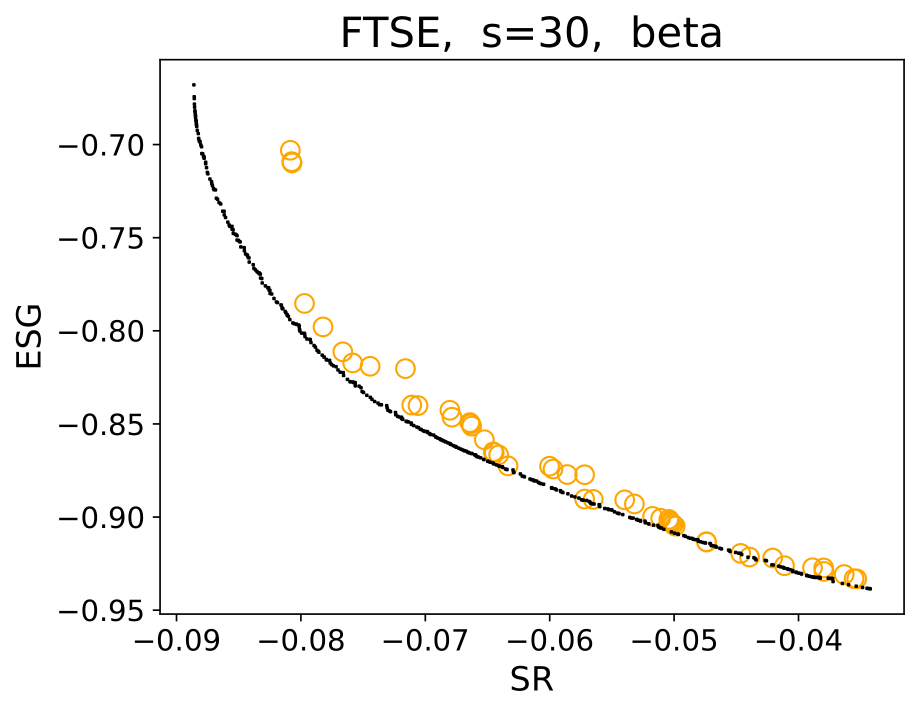}}
	\caption{Pareto front reconstruction obtained, on selected problems, with \texttt{SFSD} initialized using \texttt{MOHyb}, \texttt{NSGA-II}, \texttt{NSMA} and the \rev{linear} scalarization method; the result of vanilla \texttt{NSGA-II} is also reported. The small black dots represent the reference fronts.}
	\label{fig::SFSD_comp_all}
\end{figure}

The Pareto front reconstructions of the DTS1 mean-variance problem with $s=2$ are of particular interest: the \rev{linear} scalarization method is unable to obtain any of the solutions belonging to the central portion of the Pareto front. This behavior is in line with the discussion in section \ref{sec::limits_scalarization}, as some efficient solutions were not found not only because few values of the scalarization parameter $\lambda$ were tested, but also because some portion of the Pareto front might have been impossible to reach for any value of $\lambda$. 
In the S\&P SR-Skew problem with beta constraints and $s=10$, we observe that, without \texttt{SFSD}, the genetic algorithm \texttt{NSGA-II} struggles to obtain a good approximation of the Pareto front, even when running for 6 minutes, as in this problem not using the information of the gradient massively deteriorates the performances. This result confirms once more how \texttt{SFSD} is crucial, as \texttt{NSGA-II} with \texttt{SFSD} on the other hand is capable of obtaining a good front approximation.
  
In problems where the maximal cardinality $s$ is large we observe that \texttt{SFSD} with \texttt{MOHyb} is the most effective approach, as it is capable of reaching a more accurate approximation of the Pareto front. In particular, in the FTSE SR-ESG problem with beta constraints and $s=30$, it is able of obtaining the best Pareto front reconstruction among the methods considered. On the other hand, \texttt{SFSD} equipped with \texttt{NSGA-II} or \texttt{NSMA} obtains better results when $s$ is low. These numbers are in line with the observations reported in section \ref{subsec::phase_1}.

\subsection{Overall evaluation}
As a final step, we show the performance of \texttt{SFSD} (initialized either with \texttt{MOHyb} or \texttt{NSMA}) compared to state-of-the-art approaches, namely the \rev{linear} scalarization and \texttt{NSGA-II}, in the full benchmark of 442 problems \rev{\label{rev1.2.c}involving two or more objective functions}. Given the large number of problems, we make use of performance profiles \cite{Dolan2002} to summarize the outcome of the experiments. 

\rev{\label{rev2_metrics}We consider standard metrics in our comparison: \textit{purity}, \textit{$\Gamma$-spread} \cite{custodio11} and \textit{hypervolume} \cite{zitzler98}. In particular, the \textit{purity} metric indicates the ratio of the number of non-dominated points that a solver obtained w.r.t.\ its competitors over the number of the points produced by that solver, thus serving as a quality measure of the generated front. Clearly, a higher value is related to a better performance. The \textit{$\Gamma$-spread} metric measure how well the generated front is distributed in the objectives space: it is defined as the maximum $\ell_\infty$ distance between adjacent points of the Pareto front. As opposed to the \textit{purity}, low values for the \textit{$\Gamma$-spread} metric are associated with good performance. Finally, the \textit{hypervolume} metric calculates the area/volume which is dominated by the provided set of solutions with respect to a reference point. The latter is chosen so that each of its coordinates is slightly worse than the worst value obtained by any compared solver on the corresponding objective. Similar to \textit{purity}, higher values for the \textit{hypervolume} metric mean better performance.}

The setup of the experiment is the same of Section \ref{sec:sfsd_impact}. We split our analysis again by separately considering the 210 problems with $s\le10$, and the remaining 232 with $s>10$. 

\begin{figure}
	\centering
	\subfloat{\includegraphics[width=0.32\textwidth]{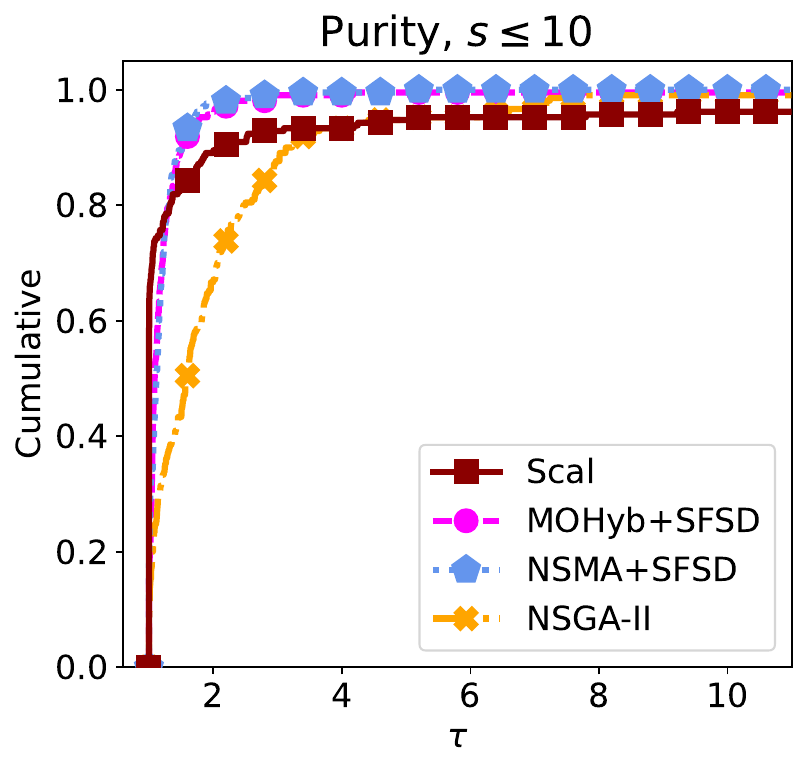}}
	\hfil
	\subfloat{\includegraphics[width=0.32\textwidth]{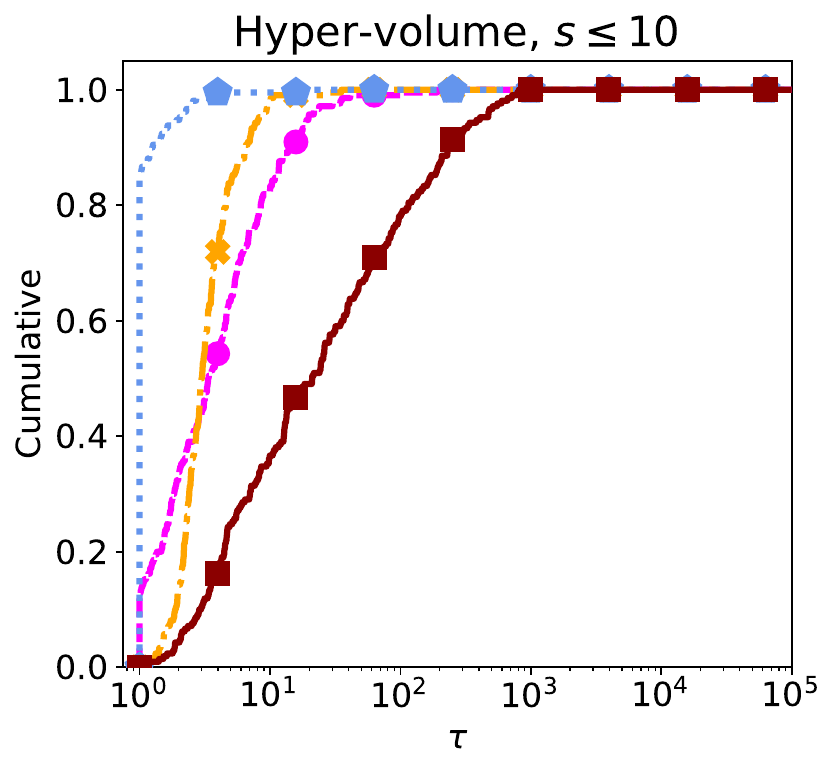}}
	\hfil
	\subfloat{\includegraphics[width=0.32\textwidth]{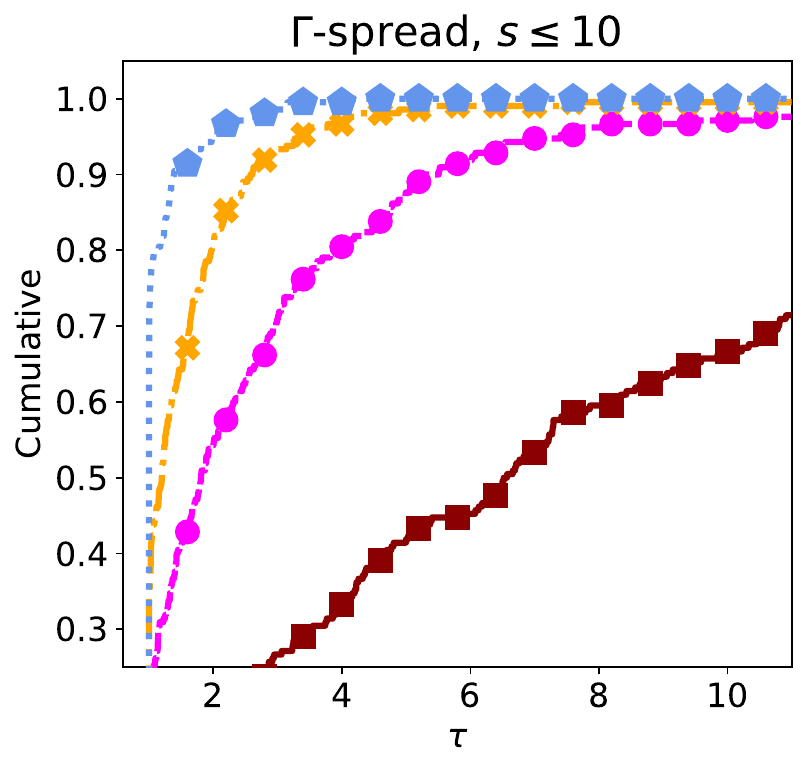}}
	\\
	\subfloat{\includegraphics[width=0.32\textwidth]{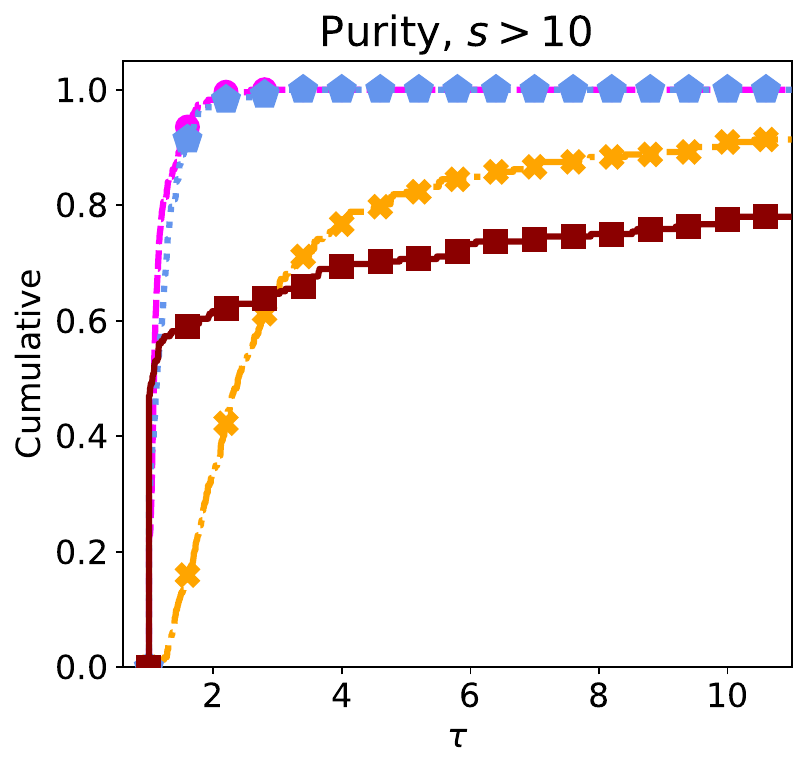}}
	\hfil
	\subfloat{\includegraphics[width=0.32\textwidth]{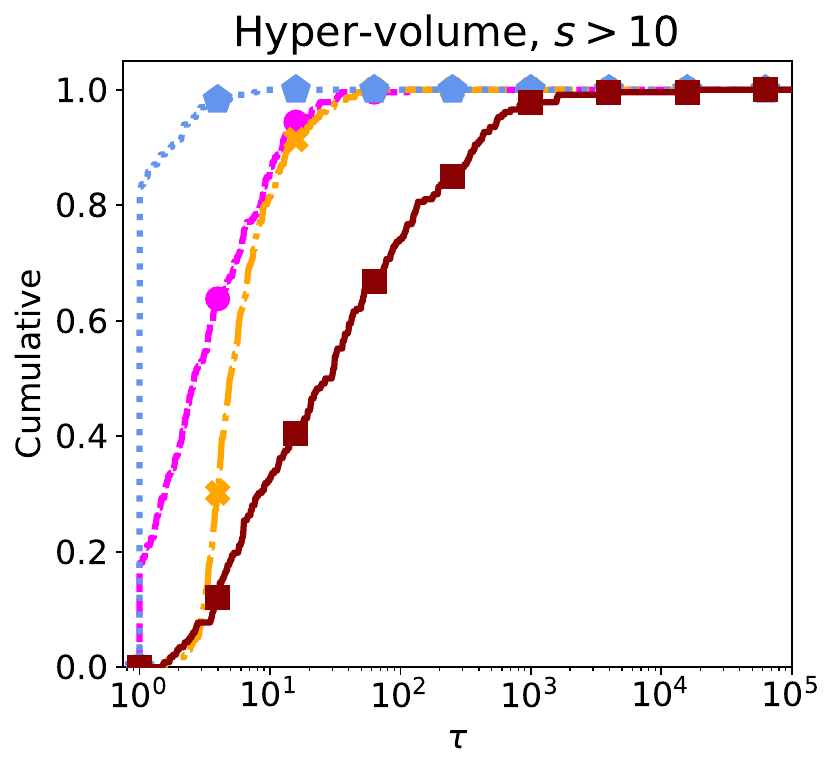}}
	\hfil
	\subfloat{\includegraphics[width=0.32\textwidth]{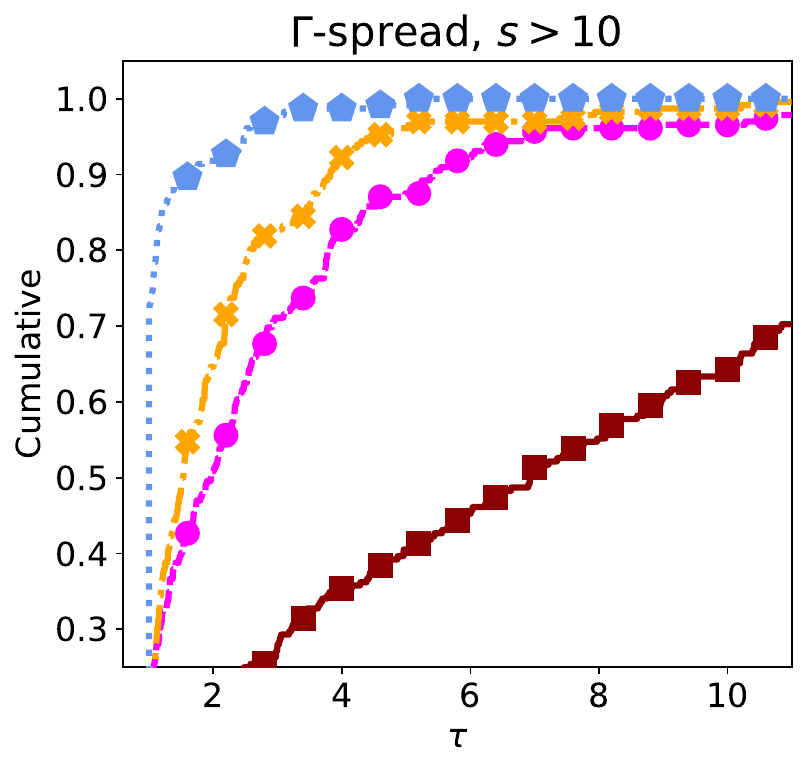}}
	\caption{Performance profiles of \textit{purity}, \textit{Hyper-volume} and \textit{$\Gamma$-spread} for \texttt{MOHyb} and \texttt{NSMA}, with \texttt{SFSD}, and for the \rev{linear} scalarization approach and vanilla \texttt{NSGA-II} \rev{on the full benchmark of 442 problems involving two or more objective functions}. We analyze, in the first row, problems with $s \leq 10$; in the second row, problems with $s > 10$.}
	\label{fig::PP}
\end{figure}

We show in Figure \ref{fig::PP} the outcome of this analysis. We can observe that \texttt{SFSD} clearly outperforms vanilla \texttt{NSGA-II} across all metrics. While both initialization strategies appear to provide solid purity values, the employment of \texttt{NSMA} seems to be a more efficient choice overall. Somewhat unsurprisingly, the \rev{linear} scalarization method falls far behind the other approaches, exhibiting once again its limitations in portfolio optimization. 

\section{Conclusions}
\label{sec::conclusions}

In this paper, we studied solution approaches for the multi-objective portfolio selection problem with sparsity constraints. After pointing out some serious limitations of the scalarization strategy, we proposed a composite framework, combining and suitably adapting to constrained problems some state-of-the-art methods, that is specifically designed to retrieve a well spanned front of efficient solutions.

From the experimental analysis reported in this manuscript, it is clear that the proposed Sparse Front Steepest Descent method is a very effective algorithm for completing the Pareto front of possible investment portfolios, when both convex and non-convex \rev{\label{rev1.2.d}(multiple)} objective functions are considered, compared to both pure evolutionary and \rev{linear} scalarization methods. While we found that feeding the method with starting solutions found by the memetic \texttt{NSMA} procedure apparently leads to the best results, we expect that using starting points produced by the combination of different strategies could even bring better Pareto front approximations, even though at the cost of a larger computational burden.

\rev{Future work may focus on integrating discrete optimization strategies into a two-stage approach to alternately identify supported and unsupported solutions, similarly to what is done in \cite{PRZYBYLSKI2010149}. Moreover, a numerical investigation of algorithmic frameworks with performance guarantees for approximating the Pareto front, inspired by the work in \cite{Papadimitriou2000}, could be of interest.}

\backmatter

%\bmhead{Supplementary information}

%If your article has accompanying supplementary file/s please state so here. 
%
%Authors reporting data from electrophoretic gels and blots should supply the full unprocessed scans for key as part of their Supplementary information. This may be requested by the editorial team/s if it is missing.
%
%Please refer to Journal-level guidance for any specific requirements.

%\bmhead{Acknowledgements}

%Acknowledgements are not compulsory. Where included they should be brief. Grant or contribution numbers may be acknowledged.
%
%Please refer to Journal-level guidance for any specific requirements.

\section*{Declarations}

\subsubsection*{Code Availability}
The implementation code of the methodologies proposed in this paper can be found at \href{https://github.com/dadoPuccio/MO-Portfolio}{github.com/dadoPuccio/MO-Portfolio}.

\subsubsection*{Conflict of interest}
The authors declare that they have no conflict of interest.

%Some journals require declarations to be submitted in a standardised format. Please check the Instructions for Authors of the journal to which you are submitting to see if you need to complete this section. If yes, your manuscript must contain the following sections under the heading `Declarations':
%
%\begin{itemize}
%\item Funding
%\item Conflict of interest/Competing interests (check journal-specific guidelines for which heading to use)
%\item Ethics approval and consent to participate
%\item Consent for publication
%\item Data availability 
%\item Materials availability
%\item Code availability 
%\item Author contribution
%\end{itemize}
%
%\noindent
%If any of the sections are not relevant to your manuscript, please include the heading and write `Not applicable' for that section. 

%%===================================================%%
%% For presentation purpose, we have included        %%
%% \bigskip command. Please ignore this.             %%
%%===================================================%%
%\bigskip
%\begin{flushleft}%
%Editorial Policies for:
%
%\bigskip\noindent
%Springer journals and proceedings: \url{https://www.springer.com/gp/editorial-policies}
%
%\bigskip\noindent
%Nature Portfolio journals: \url{https://www.nature.com/nature-research/editorial-policies}
%
%\bigskip\noindent
%\textit{Scientific Reports}: \url{https://www.nature.com/srep/journal-policies/editorial-policies}
%
%\bigskip\noindent
%BMC journals: \url{https://www.biomedcentral.com/getpublished/editorial-policies}
%\end{flushleft}
%
\begin{appendices}
\section{Further Details on Descent Methods}
\label{app::descent_methods}

In this appendix, we report some description and theoretical properties of descent algorithms mentioned throughout the paper. Although the key concepts nearly are the same as those exposed in the referenced papers, in the setting here considered we dealt with new constraints and, thus, some adjustments are needed in the theory of the presented approaches.

\subsection{Multi-objective Sparse Penalty Decomposition}

The scheme of the \textit{Multi-objective Sparse Penalty Decomposition} (\texttt{MOSPD}) \cite{Lapucci2022apenalty} can be found in Algorithm \ref{alg::MOSPD}. Here below, we list the key steps of the approach.

\begin{algorithm}
	\caption{\textit{Multi-objective Sparse Penalty Decomposition} (\texttt{MOSPD})}\label{alg::MOSPD}
	\begin{algorithmic}[1]
		\Require $F:\mathbb{R}^n \rightarrow \mathbb{R}^m$, $x_0 = y_0 \in \Omega = \Omega_c \cap \{x \in \mathbb{R}^n \mid \|x\|_0 \le s\}$, $\tau_0 > 0$, $\sigma > 1$, $\{\varepsilon_k\} \subseteq \mathbb{R}_+$ s.t. $\varepsilon_k \rightarrow \infty$.
		\For{$k = 0,1,\ldots$}
		\State Let $Q^{\tau_k}(x, y_k) = F(x) + \mathbf{1}_m\frac{\tau_k}{2}\|x-y_k\|^2$ \label{line::def_Q}
		\State $\ell = 0$
		\State $x_\text{trial} =$ \texttt{MOPG}($Q^{\tau_k}(\cdot, y_k)$, $x_k$, $\varepsilon_k$) \label{line::start_test}
		\If{$Q^{\tau_k}(x_\text{trial}, y_k) \le F(x_0)$}
		\State Set $u_0, v_0 = x_k, y_k$
		\Else
		\State Set $u_0, v_0 = x_0, y_0$ \label{line::end_test}
		\EndIf
		\While{$\min\limits_{d \in \mathcal{D}_{\Omega_c}(u_\ell)}\max\limits_{j \in \{1,\ldots,m\}}\nabla_uq^{\tau_k}_j(u_\ell, v_\ell)^\top d + \frac{1}{2}\|d\|^2 < -\varepsilon_k$ \label{line::start_alt}}
		\State $u_{\ell+1} =\ $\texttt{MOPG}($Q^{\tau_k}(\cdot, v_\ell)$, $u_\ell$, $\varepsilon_k$) \label{line::MOPG}
		\State $v_{\ell+1} = \argmin\limits_{v \in \mathbb{R}^n, v\ge\mathbf{0}_n, \|v\|_0 \le s}\|u_{\ell+1} - v\|^2$ \label{line::proj}
		\State $\ell = \ell + 1$ \label{line::end_alt}
		\EndWhile
		\State $\tau_{k+1} = \sigma\tau_k$
		\State $x_{k+1} = u_\ell$
		\State $y_{k+1} = v_\ell$
		\EndFor\\
		\Return sequence $\{(x_k, y_k)\}$
	\end{algorithmic}
\end{algorithm}

\begin{itemize}
	\item The algorithm starts at a point $x_0$ which is feasible for problem \eqref{eq::port-prob}. Here, we introduce the notation $\Omega_c = \{x \in \mathbb{R}^n \mid Ax \le b,\; \mathbf{1}_n^\top x = 1,\; x \ge \mathbf{0}_n\}$ to indicate the feasible set induced by the convex constraints in problem \eqref{eq::port-prob}.
	\item As already anticipated in section \ref{subsec::phase1}, the method aims to find, at each iteration $k$, a pair $(x_{k+1}, y_{k+1})$ such that $y_{k+1}$ is the projection of $x_{k+1}$ onto the sparse set and $x_{k+1}$ is approximately Pareto stationary w.r.t.\ $Q^{\tau_k}(\cdot, y_{k+1})$ with Pareto stationarity approximation degree $\varepsilon_k \to 0$ and $\tau_k \to \infty$ for $k \to \infty$. The goal is attained employing the \textit{Multi-objective Projected Gradient} (\texttt{MOPG}) \cite{Drummond2004} method in an \textit{alternate minimization scheme}.
	\item Before entering in the \textit{alternate minimization loop}, a test (lines \ref{line::start_test}-\ref{line::end_test}) is carried out to ensure that every next iterate will stay inside an appropriate level set; more information on this scheme part can be found in \cite{Lapucci2022apenalty}.
	\item In an alternate minimization fashion, in lines \ref{line::start_alt}-\ref{line::end_alt} we employ a run of \texttt{MOPG} on $Q^{\tau_k}$ w.r.t. $x$ (line \ref{line::MOPG}) and a projection operator onto the sparse feasible set (line \ref{line::proj}); when the resulting solution is approximately Pareto-stationary on $Q^{\tau_k}$ w.r.t.\ $x$ (line \ref{line::start_alt}), we then increase $\tau_k$ and decrease the Pareto stationarity approximation degree $\varepsilon_k$ for the next iteration.
\end{itemize}

In order to state the convergence properties of \texttt{MOSPD}, we need the following assumption.
\begin{assumption}
	\label{ass::bounded_set}
	$F$ has bounded level sets in the multi-objective sense, i.e., the set $\mathcal{L}_F(z) = \{x \in \Omega \mid F(x) \le z\}$ is bounded for any $z \in \mathbb{R}^m$.
\end{assumption}

Moreover, we have to state a technical lemma.
\begin{lemma}
	\label{lem::tech}
	Let $D \subset \mathbb{R}^n$ be a (\rev{possibly} non-convex) set such that if $d \in D$ then $td \in D$ for all $t \in [0, 1]$, $A \in \mathbb{R}^{p \times n}$, $x \ge \mathbf{0}_n$ and $F: \mathbb{R}^n \rightarrow \mathbb{R}^m$ be a continuously differentiable function. If
	\begin{equation}
		\label{eq::first_lemma}
		\begin{aligned}
			\min_{d \in D}\;&\max_{j \in \{1,\ldots, m\}} \nabla f_j(x)^\top d + \frac{1}{2}\|d\|^2 = 0
			\\\text{s.t. }& A(x+d) \le b, \quad \mathbf{1}_n^\top(x+d) = 1, \quad x+d \ge \mathbf{0}_n,
		\end{aligned}
	\end{equation}
	then
	\begin{equation}
		\label{eq::second_lemma}
		\begin{aligned}
			\min_{d \in D}\;&\max_{j \in \{1,\ldots, m\}} \nabla f_j(x)^\top d + \frac{1}{2}\|d\|^2 = 0
			\\\text{s.t. }& a_i^\top d \le 0 \quad \forall i: a_i^\top x=b \\& \mathbf{1}_n^\top d = 0, \quad d_i \ge 0 \quad \forall i: x_i = 0.
		\end{aligned}
	\end{equation}
\end{lemma}
\begin{proof}
    \rev{By contradiction, let us assume that \eqref{eq::second_lemma} is not satisfied, and let
    \begin{equation*}
        \begin{aligned}
            \tilde{d} \in \argmin_{d \in D}\;&\max_{j \in \{1,\ldots, m\}} \nabla f_j(x)^\top d + \frac{1}{2}\|d\|^2
            \\\text{s.t. }& a_i^\top d \le 0 \quad \forall i: a_i^\top x=b \\& \mathbf{1}_n^\top d = 0, \quad d_i \ge 0 \quad \forall i: x_i = 0.
        \end{aligned}
    \end{equation*}}

    First, note that the constraints (excluding the set $D$) in problem \eqref{eq::second_lemma} constitute the feasible direction set $\mathcal{D}_{\Omega_c}(x)$. Let us denote as $\tilde{\mathcal{D}}(x)$ the set formed by the related constraints in problem \eqref{eq::first_lemma}. It is easy to prove that $\tilde{\mathcal{D}}(x) \subseteq \mathcal{D}_{\Omega_c}(x)$, $\mathbf{0}_n \in \tilde{\mathcal{D}}(x)$ and, finally, both sets are convex. \rev{Thus, there exists $0 < t < 1$ sufficiently small such that $t\tilde{d} \in \tilde{\mathcal{D}}(x)$. Moreover, by definition of the set $D$, $\tilde{d} \in D$ implies that $t\tilde{d} \in D$, concluding that $t\tilde{d}$ is feasible for problem \eqref{eq::first_lemma}. But, by definition of $\tilde{d}$ and $t$, $\max_{j \in \{1,\ldots, m\}}\nabla f_j(x)^\top(t\tilde{d}) + \frac{1}{2}\|t\tilde{d}\|^2 = t^2(\frac{\max_{j \in \{1,\ldots, m\}}\nabla f_j(x)^\top\tilde{d}}{t} + \frac{1}{2}\|\tilde{d}\|^2) < 0$. Combining the last two statements, we contradict \eqref{eq::first_lemma} and, thus, we get the thesis.}
\end{proof}

We are finally ready to state the convergence properties of \texttt{MOSPD}.
\begin{lemma}
	Let Assumption \ref{ass::bounded_set} hold and $\{(x_k, y_k)\}$ be the sequence generated by Algorithm \ref{alg::MOSPD}. Then, $\{(x_k, y_k)\}$ admits cluster points. Moreover, let $(\bar{x}, \bar{y})$ a limit point of the sequence, i.e., there exists $K \subseteq \{0,1,\ldots\}$ such that $(x_k, y_k) \rightarrow (\bar{x}, \bar{y})$ for $k \rightarrow \infty$, $k \in K$. We have the following conditions hold:
	\begin{enumerate}
		\item $\bar{x}$ is feasible for problem \eqref{eq::port-prob};
		\item $\bar{x}$ is Pareto stationary in a subspace associated to one of its super support set for problem \eqref{eq::port-prob}.
	\end{enumerate}
\end{lemma}
\begin{proof}
	The theses straightforwardly follow as in \cite[Propositions 4.5, 4.7]{Lapucci2022apenalty}, respectively, where Lemma \ref{lem::tech} can be used in place of Lemma 4.6 of the referenced paper.
\end{proof}
\begin{remark}
	The stationarity condition of \texttt{MOSPD} has been referred to as \textit{Multi-objective Lu-Zhang first-order} (MOLZ) conditions \cite{Lapucci2022apenalty}. In the cited paper, MOLZ conditions are shown to be weaker than standard Pareto stationarity for problem \eqref{eq::port-prob} (see Lemma \ref{lem::par-stat}): while the latter one implies MOLZ conditions, a point must satisfy MOLZ conditions for every associated super support set to be Pareto stationary.
\end{remark}

\subsection{Multi-objective Iterative Hard Thresholding}

The scheme of the \textit{Multi-objective Iterative Hard Thresholding} (\texttt{MOIHT}) \cite{Lapucci2024cardinality} is reported in Algorithm \ref{alg::MOIHT}. 

\begin{algorithm}
	\caption{\textit{Multi-objective Iterative Hard Thresholding} (\texttt{MOIHT})}\label{alg::MOIHT}
	\begin{algorithmic}[1]
		\Require $F:\mathbb{R}^n \rightarrow \mathbb{R}^m$, $x_0 \in \Omega$, $L > \max_{j \in \{1,\ldots,m\}}L(\nabla f_j)$.
		\State $k = 0$
		\While{$x_k$ is not $L$-stationary for problem \eqref{eq::port-prob}}
		\State Compute 
		\begin{equation}
			\label{eq::dir_l}
			\begin{aligned}
				v_L(x_k) \in \argmin_{d \in \mathbb{R}^n}\;&\max_{j \in \{1,\ldots,m\}} \nabla f_j(x_k)^\top d + \frac{L}{2}\|d\|^2
				\\\text{s.t. }& A(x_k+d) \le b, \quad \mathbf{1}_n^\top(x_k+d) = 1 \\& \|x_k + d\|_0 \le s, \quad x_k+d \ge \mathbf{0}_n.
			\end{aligned}
		\end{equation}
		\State $x_{k+1} = x_k + v_L(x_k)$
		\State Let $k = k + 1$
		\EndWhile\\
		\Return sequence $\{x_k\}$
	\end{algorithmic}
\end{algorithm}

For a brief description of the method mechanisms, we need first an assumption on the objective functions gradients.
\begin{assumption}
	\label{ass::lipschitz}
	For all $j \in \{1,\ldots,m\}$, $\nabla f_j$ is Lipschitz-continuous over $\mathbb{R}^n$ with constant $L(\nabla f_j)$, i.e., $\|\nabla f_j(x) - \nabla f_j(y)\| \le L(\nabla f_j)\|x - y\|$ for all $x, y \in \mathbb{R}^n$.
\end{assumption}

The algorithm starts from an initial feasible point $x_0 \in \Omega$ and iteratively updates the point until it is $L$-stationary, with $L$ being an input of the approach. At each iteration, the update is performed by solving problem \eqref{eq::dir_l}. Note that, by definition of the feasible set of problem \eqref{eq::dir_l}, all the iterates generated by \texttt{MOIHT} are feasible for problem \eqref{eq::port-prob}. Moreover, managing the cardinality constraint in problems like \eqref{eq::dir_l} is actually a practical operation to be carried out with mixed-integer quadratic solvers.

The convergence properties directly follows from the theoretical analysis of \texttt{MOIHT} in \cite{Lapucci2024cardinality}: indeed, the presence of constraints other than the cardinality one does not spoil any reported proof.
\begin{lemma}[{\cite[Proposition 4.1]{Lapucci2024cardinality}}]
	Let Assumptions \ref{ass::bounded_set}-\ref{ass::lipschitz} hold and $\{x_k\}$ be the sequence generated by Algorithm \ref{alg::MOIHT} with constant $L > \max_{j \in \{1,\ldots, m\}}L(\nabla f_j)$. Then, the sequence admits cluster points, each one being $L$-stationary for problem \eqref{eq::port-prob}.
\end{lemma}

\subsection{Sparse Front Steepest Descent}

The \textit{Sparse Front Steepest Descent} (\texttt{SFSD}) is already described in section \ref{subsec::SFSD}. Moreover, by definition of the constrained descent directions (problems \eqref{eq::common-dir}-\eqref{eq::partial-dir}) and the fact that the first tentative step size in the Armijo-type line searches is $\alpha_0 = 1$, most of theoretical properties of the approach, including the finite termination of the line searches, are already stated and proved in \cite{Lapucci2024cardinality, LAPUCCI2023242}. Here, we only report the proof of Proposition \ref{prop::conv-SFSD}, which must take into account the new feasible set $\Omega$.

\noindent \textbf{Proposition \ref{prop::conv-SFSD}.} \textit{Let us assume that $X^0_J$ is a set of mutually nondominated points and there exists $x_0^J \in X^0_J$ such that the set $\widehat{\mathcal{L}}_F(x_0^J) = \bigcup_{j=1}^m\{x \in \Omega \mid f_j(x) \le f_j(x_0^J)\}$ is compact. Moreover, let $X^k_J$ be the sequence of sets of nondominated points, associated with the super support set $J$, produced by Algorithm \ref{alg::SFSD}, and $x_{j_k}^J$ be a linked sequence. Then, the latter admits accumulation points, each one being Pareto stationary in the subspace associated with $J$ for problem \eqref{eq::port-prob}.}
\begin{proof}
	By instructions of the algorithms, each linked sequence $\{x_{j_k}^J\}$ can be seen as generated by applying front steepest descent \cite{lapucci2024effectivefrontdescentalgorithmsconvergence, LAPUCCI2023242} to the problem
	\begin{equation*}
		\begin{aligned}
			\min_{x \in \mathbb{R}^n}\;&F(x)
			\\\text{s.t. }& Ax \le b, \quad \mathbf{1}_n^\top x = 1 \\& x_{\bar{J}} = \mathbf{0}_{|\bar{J}|}, \quad x \ge \mathbf{0}_n,
		\end{aligned}
	\end{equation*}
	employing as descent directions the constrained ones proposed in equations \eqref{eq::common-dir}-\eqref{eq::partial-dir} and initial step size $\alpha_0=1$.
	Thus, following the same reasoning used for proof of \cite[Propositiom 3.4]{LAPUCCI2023242}, we can straightforwardly prove that each accumulation point $\bar{x}$ of the linked sequence $\{x_{j_k}^J\}$ is such that $\theta_{J_{x_c}}(\bar{x}) = 0$. Recalling now Lemma \ref{lem::tech}, we finally get that $\bar{x}$ is Pareto stationary in the subspace associated with $J$ for problem \eqref{eq::port-prob}.
\end{proof}

%\label{app::proofs}
%
%An appendix contains supplementary information that is not an essential part of the text itself but which may be helpful in providing a more comprehensive understanding of the research problem or it is information that is too cumbersome to be included in the body of the paper.

%%=============================================%%
%% For submissions to Nature Portfolio Journals %%
%% please use the heading ``Extended Data''.   %%
%%=============================================%%

%%=============================================================%%
%% Sample for another appendix section			       %%
%%=============================================================%%

%% \section{Example of another appendix section}\label{secA2}%
%% Appendices may be used for helpful, supporting or essential material that would otherwise 
%% clutter, break up or be distracting to the text. Appendices can consist of sections, figures, 
%% tables and equations etc.

\end{appendices}

%%===========================================================================================%%
%% If you are submitting to one of the Nature Portfolio journals, using the eJP submission   %%
%% system, please include the references within the manuscript file itself. You may do this  %%
%% by copying the reference list from your .bbl file, paste it into the main manuscript .tex %%
%% file, and delete the associated \verb+\bibliography+ commands.                            %%
%%===========================================================================================%%

\bibliography{sn-article}% common bib file
%% if required, the content of .bbl file can be included here once bbl is generated
%%\input sn-article.bbl

\end{document}